\documentclass[a4paper,11pt]{article}

\usepackage[margin=1.0in]{geometry}

\usepackage{hyperref}
\usepackage{algorithm,algorithmic}
\usepackage{amsmath}
\usepackage{amsthm}
\usepackage{amsfonts}
\usepackage{latexsym}
\usepackage{amssymb}
\usepackage{xspace}
\usepackage{color}
\usepackage{fullpage}
\usepackage{epsfig}
\usepackage{tikz}
\usepackage{array}
\usepackage{comment}

\newtheorem{Thm}{Theorem}

\newtheorem{prop}[Thm]{Proposition}

\def\x{{\mathbf{x}}}
\def\u{{\mathbf{u}}}
\def\v{{\mathbf{v}}}
\def\z{{\mathbf{z}}}
\def\w{{\mathbf{w}}}
\def\y{{\mathbf{y}}}
\def\b{{\mathbf{b}}}
\def\X{{\mathbf{X}}}

\def\A{{\mathbf{A}}}
\def\M{{\mathbf{M}}}
\def\I{{\mathbf{I}}}
\def\B{{\mathbf{B}}}

\def\W{{\mathbf{W}}}

\def\bmu{{\mathbf{\mu}}}

\def\mD{{\mathcal D}}
\def\reals{{\mathcal R}}

\newcommand{\ignore}[1]{}

\def\reals{{\mathbb R}}

\def\mA{{\mathcal A}}

\def\bold0{\mathbf{0}}

\newcommand\E{\mbox{\bf E}}

\def\w{\mathbf{w}}
\def\x{\mathbf{x}}

\newtheorem{theorem}{Theorem}[section]
\newtheorem{definition}{Definition}[section]
\newtheorem{lemma}{Lemma}[section]

\title{Fast and Simple PCA via Convex Optimization}
\author{Dan Garber\\  
\small{Technion - Israel Institute of Technology} \\
\small{dangar@tx.technion.ac.il} 
\and Elad Hazan \\
\small{Princeton University} \\ 
\small{ehazan@cs.princeton.edu}}

\date{}

\begin{document} 
\maketitle

\begin{abstract}
The problem of principle component analysis (PCA) is traditionally solved by spectral or algebraic methods. We show how computing the leading principal component could be reduced to solving a \textit{small} number of well-conditioned {\it convex} optimization problems. This gives rise to a new efficient method for PCA based on recent advances in stochastic methods for convex optimization. 

In particular we show that given a $d\times d$ matrix $\X = \frac{1}{n}\sum_{i=1}^n\x_i\x_i^{\top}$ with top eigenvector $\u$ and top eigenvalue $\lambda_1$ it is possible to:
\begin{itemize}
\item compute a unit vector $\w$ such that $(\w^{\top}\u)^2 \geq 1-\epsilon$ in $\tilde{O}\left({\frac{d}{\delta^2}+N}\right)$ time, where $\delta = \lambda_1 - \lambda_2$ and $N$ is the total number of non-zero entries in $\x_1,...,\x_n$,

\item compute a unit vector $\w$ such that $\w^{\top}\X\w \geq \lambda_1-\epsilon$ in $\tilde{O}(d/\epsilon^2)$ time.
\end{itemize}
To the best of our knowledge, these bounds are the fastest to date for a wide regime of parameters. These results could be further accelerated when $\delta$ (in the first case) and $\epsilon$ (in the second case) are smaller than $\sqrt{d/N}$.

\end{abstract}

\section{Introduction}

Since its introduction by Pearson \cite{Pearson} and Hotelling \cite{Hotelling}, the principle component analysis technique of finding the subspace of largest variance in the data has become ubiquitous in unsupervised learning, feature generation and data visualization. 

For data given as a set of $n$ vectors in $\mathbb{R}^d$, $\x_1,...,\x_n$, denote by $\X$ the normalized covariance matrix $\X = \frac{1}{n}\sum_{i=1}^n\x_i\x_i^{\top}$. The PCA method finds the $k$-dimensional subspace such that the projection of the data onto the subspace has largest variance. Formally, let $\W \in \reals^{d \times k}$ be an orthogonal projection matrix, PCA  could be formalized as the following optimization problem.

\begin{equation}\label{eq:pca}
\max_{\W \in \reals^{d \times k} \ , \W^T \W = \I} \| \X\W \|_F^2 , \tag{PCA}
\end{equation}

where $\Vert\cdot\Vert_F$ is the Frobenius norm. Note that the above optimization problem is an inherently non-convex optimization problem even for $k=1$.
Henceforth  we focus on the problem of finding with high precision only the leading principal component, i.e. on the case $k=1$. 

Finding the leading principal component could be carried out using matrix factorization techniques in time $O(nd^2 + d^3)$, by explicitly computing the matrix $\X$ and then computing its singular value decomposition ($\textrm{SVD}$). However this requires super-linear time and potentially $O(d^2)$ space. 

Since super-linear times are prohibitive for large scale machine learning problems, approximation methods have been the focus of recent literature. Iterative eigenvector methods, such as the Lanczos method or Power method \cite{MatrixCompBook}, can be applied  without explicitly computing the matrix $\X$. These latter methods require the data to be well-conditioned, and the spectral gap, i.e., the distance between largest and second largest eigenvalues of $\X$, to be bounded away from zero. 
  
If we let $\delta > 0$ denote this spectral gap, then the Power and Lanczos methods requires roughly $(\lambda_1/\delta)^{-\gamma}$ passes over the entire data to compute the leading PC which amounts to $\tilde{O}((\lambda_1/\delta)^{\gamma}N)$ total time, where $\gamma=1$ for the Power method and $\gamma = 1/2$ for the Lanczos method, and $N$ it total number of non-zero entries in the vectors $\x_1,...,\x_n$.

Thus, iterative methods replace the expensive super-linear operations of matrix factorization by linear-time operations, but require several passes over the data. Depending on the spectral gap of $\X$, the latter methods may also be prohibitively expensive. This motivates the study of methods that retain the best from both worlds: linear time per iteration, as well as doing only a small number (i.e., at most logarithmic in the natural parameters of the problem) of passes over the entire data. 

In the convex optimization world, such hybrid results with simple iterations coupled with only a few passes over the data were obtained in recent years \cite{Zhang13, Zhang13b, KonecnyR13}.  Recently Shamir \cite{Shamir14, Shamir15} made headway in the non-convex spectral world by suggesting a stochastic optimization method for the problem, that is based Oja's randomized Power method  \cite{Oja82,DasguptaFreund} with an application of the variance reduction principle demonstrated in \cite{Zhang13}. 

Shamir's algorithm runs in total time of $\tilde{O}(\delta^{-2}d + N)$, assuming the availability of a unit vector $\w$ such that $(\w^{\top}\u)^2 = \Omega(1)$, called a ``warm start". Finding such a warm start vector $\w$ could take as much as $\tilde{O}(\sqrt{\lambda_1/\delta}\delta^{-2}d)$ time using existing methods \footnote{Finding a ``warm start" could be carried out either by applying iterative methods such as Power and Lanczos methods to the entire data or by applying them only to a small random sample of the data. Since we want to be overall competitive with these methods, we focus on the second option.}, hence the total running time becomes $\tilde{O}(\sqrt{\lambda_1/\delta}\delta^{-2}d + N) $. Quite remarkably, Shamir's result separates the dependency on the gap $\delta$ and the data size $N$.

Analyzing the Power method in case of stochastic updates, as done in \cite{DasguptaFreund, Shamir14}, 
is much more intricate than analyzing stochastic gradient algorithms for convex optimization, and indeed both the analysis of Oja's algorithm in \cite{Oja82} and Shamir's algorithm in \cite{Shamir14, Shamir15} are quite elaborate.

In this paper we continue the search for efficient algorithms for PCA. 
Our main contribution, at a high-level, is in showing that the problem of computing the largest eignevector of a positive semidefinite matrix could be reduced to solving only a poly-logarithmic number of well conditioned unconstrained smooth and strongly convex optimization problems. These well-conditioned convex optimization problems could then be solved by a variety of algorithms that are available in the convex and stochastic optimization literature, and in particular, the recent algorithmic advances in stochastic convex optimization, such as the results in \cite{Zhang13, Zhang13b, KonecnyR13,Harchaoui15, Kakade15}.
As a result, we derive algorithms that in terms of running time, their worst case running time starting from a \textbf{``cold-start"} is equivalent or better than that of the algorthm of Shamir initialized with a \textbf{``warm-start"} . 

\subsection{Problem setting and main results}
Assume we are given a set of $n$ vectors in $\mathbb{R}^d$, $\x_1,...,\x_n$, where $n>d$ and let us denote by $\X$ the normalized covariance matrix $\X = \frac{1}{n}\sum_{i=1}^n\x_i\x_i^{\top}$. Assume further that $\X$ has an eigengap of $\delta$, i.e. $\lambda_1(\X)-\lambda_2(\X) = \delta$, and w.l.o.g. that $\Vert{\x_i}\Vert \leq 1$ for all $i\in[n]$. Our goal is to find a unit vector $\w$ such that 
\begin{eqnarray*}
(\w^{\top}\u)^2 \geq 1-\epsilon ,
\end{eqnarray*}

where $\u$ is the leading eigenvector of $\X$ and $\epsilon$ is the desired accuracy parameter.

In the rest of the paper we denote the eigenvalues of $\X$ in descending order by $\lambda_1 \geq \lambda_2 \geq ... \geq \lambda_d$ and by $\u = \u_1, \u_2, ...,\u_d$ the corresponding eigenvectors. We denote by $N$ the total number of non-zero entries in the vectors $\x_1,...,\x_n$.

We assume we are given an estimate $\hat{\delta}$ such that 
\begin{eqnarray*}
c_1\delta \leq \hat{\delta} \leq c_2\delta ,
\end{eqnarray*}
for some universal constants $c_1,c_2$ which satisfy $c_2 - c_1 = \Theta(1)$. Note that finding such a $\hat{\delta}$ could be done in time that is logarithmic in $1/\delta$. 

The main result of this paper is the proof of the following high-level proposition.
\begin{prop}
There exists an algorithm that after solving a poly-logarithmic number of well-conditioned unconstrained convex optimization problems, returns a unit vector $\w$, such that $(\w^{\top}\u)^2 \geq 1- \epsilon$.
\end{prop}

Based on this proposition we prove the following theorems. 

\begin{theorem}\label{thm:main}
Fix $\epsilon >0, p>0$. There exists an algorithm that finds with probability at least $1-p$ a unit vector $\w$ such that $(\w^{\top}\u)^2 \geq 1-\epsilon$, in total time $\tilde{O}\left({\frac{d}{\delta^2}+N}\right)$.
\end{theorem}

\begin{theorem}\label{thm:mainAcc}
Fix $\epsilon >0, p>0$. Assume that $\delta = o(\sqrt{d/N})$. There exists an algorithm that finds with probability at least $1-p$ a unit vector $\w$ such that $(\w^{\top}\u)^2 \geq 1-\epsilon$, in total time $\tilde{O}\left({\frac{N^{3/4}d^{1/4}}{\sqrt{\delta}}}\right)$.
\end{theorem}

Throughout this work we use the notation $\tilde{O}(\cdot)$ to hide poly-logarithmic dependencies on $d, \epsilon^{-1}, p^{-1}, \delta^{-1}$.

\begin{table}[H]\label{table:prevwork:gap}
\begin{center}
  \begin{tabular}{| l | c | c |}
    \hline
   Method & Complexity  & Comments\\ \hline
   Power method  \cite{MatrixCompBook} & $\frac{\lambda_1}{\delta}N$ &\\ \hline
   Lanczos \cite{MatrixCompBook}   & $\sqrt{\frac{\lambda_1}{\delta}}N$ &\\ \hline
   VR-PCA \cite{Shamir14, Shamir15}  & $\sqrt{\frac{\lambda_1}{\delta}}\frac{d}{\delta^{2}} + N $ & \\ \hline
   Theorem \ref{thm:main} & $\frac{d}{\delta^2} + N$ & fastest when $\delta = \Omega(\sqrt{d/N})$ \\ \hline
   Theorem \ref{thm:mainAcc} ($\delta = o(\sqrt{d/N})$) & $\frac{N^{3/4}d^{1/4}}{\sqrt{\delta}}$ & fastest when $\lambda_1 = \omega(\sqrt{d/N})$, $ \delta = o(\sqrt{d/N})$ \\ \hline
  \end{tabular}
  \caption{Comparison with previous eigengap-based results. Note that the result of \cite{Shamir14, Shamir15} apply in general only from a ``warm-start", i.e. when initialized with a unit vector $\w_0$ such that $(\w_0^{\top}\u)^2 = \Omega(1)$. Finding such a warm-start could be carried out efficiently using a sample of roughly $1/\delta^2$ matrices $\x_i\x_i^{\top}$ and applying the Lanczos method to this sample.}
\end{center}
\end{table}

\subsubsection{Gap-free results}

It makes since to not only consider the problem of approximating the leading eigenvector of the matrix $\X$, but also the problem of approximating the corresponding eignevalue, that is the problem of finding a unit vector $\w$ such that
\begin{eqnarray*}
\w^{\top}\X\w \geq \lambda_1 -\epsilon.
\end{eqnarray*}

For this purpose we also prove the following theorems which are analogous to Theorems \ref{thm:main}, \ref{thm:mainAcc}.

\begin{theorem}\label{thm:main:gapfree}
Fix $\epsilon >0, p>0$. There exists an algorithm that finds with probability at least $1-p$ a unit vector $\w$ such that $\w^{\top}\X\w \geq \lambda_1 - \epsilon$, in total time $\tilde{O}\left({\frac{d}{\epsilon^2}}\right)$.
\end{theorem}

\begin{theorem}\label{thm:main:gapfreeAcc}
Fix $\epsilon = o(\sqrt{d/N}), p>0$.  There exists an algorithm that finds with probability at least $1-p$ a unit vector $\w$ such that $\w^{\top}\X\w \geq \lambda_1 - \epsilon$, in total time $\tilde{O}\left({\frac{N^{3/4}d^{1/4}}{\sqrt{\epsilon}}}\right)$.
\end{theorem}

We note that Theorem \ref{thm:main:gapfree}, in which the bound is independent of $n$, applies also in the case when $\X$ is not given explicitly, but only through a stochastic oracle that when queried, returns a rank-one matrix $\x_i\x_i^{\top}$ sampled at random from an unknown distribution $\mD$ that satisfies $\X = \E_{\x\sim\mD}[\x\x^{\top}]$. See section \ref{sec:gapfree} for more details.

\begin{table}[H]\label{table:prevwork:gapfree}
\begin{center}
  \begin{tabular}{| l | c | c |}
    \hline
   Method & Complexity  & Comments\\ \hline
   Power method  \cite{EigenvaluesApprox} & $\frac{\lambda_1}{\epsilon}N$ &\\ \hline
   Lanczos \cite{EigenvaluesApprox}   & $\sqrt{\frac{\lambda_1}{\epsilon}}N$ &\\ \hline
   Sample PM   & $\frac{\lambda_1}{\epsilon}\frac{d}{\epsilon^2}$ & \\ \hline
   Sample Lanczos   & $\sqrt{\frac{\lambda_1}{\epsilon}}\frac{d}{\epsilon^2}$ & \\ \hline
   Theorem \ref{thm:main:gapfree} & $d/\epsilon^2$ & fastest when $\epsilon = \Omega\left({\max\{\frac{d^{2/3}}{\lambda_1^{1/3}N^{1/3}}, \sqrt{d/N}\}}\right)$ \\ \hline
   Theorem \ref{thm:main:gapfreeAcc} ($\epsilon =o(\sqrt{d/N})$) & $\frac{N^{3/4}d^{1/4}}{\sqrt{\epsilon}}$ & fastest when $\lambda_1 = \omega(\sqrt{d/N})$, $\epsilon =o(\sqrt{d/N})$ \\ \hline   
  \end{tabular}
  \caption{Comparison with previous eigengap-free results. Sample PM and Sample Lanczos refer to the application of the standard Power Method and Lanczos algorithms to a random sample of size roughly $\epsilon^{-2}$, chosen uniformly at random from the matrices $\{\x_i\x_i^{\top}\}_{i=1}^n$. Such a sample suffices to approximate the top eigenvalue to precision $\epsilon$ (using standard concentration results for matrices such as the Matrix Hoeffding inequality \cite{Tropp12}). }
\end{center}
\end{table}

The rest of this paper is organized as follows. In Section \ref{sec:prelim} we give relevant preliminaries on classical methods for the computation of the largest eigenvector and related tools, and also preliminaries on convex optimization. In Section \ref{sec:PowerMethodBasedAlg} we present the core idea for our algorithms: a \textit{shrinking inverse power method} algorithm that computes the leading eigenvector after computing only a poly-logarithmic number of matrix-vector products. Based on this algorithm, in Section \ref{sec:convexEV} we present our convex optimization-based eigenvector algorithm that requires to solve only a poly-logarithmic number of well-conditioned convex unconstrained optimization problems in order to compute the largest eigenvector. In Section \ref{sec:SVRG4PCA} we combine the result of Section \ref{sec:convexEV} with recent fast stochastic methods for convex optimization to prove Theorems \ref{thm:main}, \ref{thm:mainAcc}. In Section \ref{sec:gapfree} we prove Theorems \ref{thm:main:gapfree}, \ref{thm:main:gapfreeAcc} .

\section{Preliminaries}\label{sec:prelim}

\subsection{Classical algorithms for the leading eigenvector problem}\label{sec:powerm}

\subsubsection{The Power Method}

Our main algorithmic tool for proving the convergence of our method is the classical \textit{Power Method} for approximating the leading eigenvalue and eigenvector of a positive definite matrix.

\begin{algorithm}[H]
\caption{\textsc{Power  Method}}
\label{alg:powerm}
\begin{algorithmic}[1]
\STATE Input: a positive definite matrix $\M$
\STATE Let $\w_0$ to be a random unit vector
\FOR{$t=1,2,...$}
\STATE $\w_t \gets \frac{\M\w_{t-1}}{\Vert{\M\w_{t-1}}\Vert}$
\ENDFOR
\end{algorithmic}
\end{algorithm}

In our analysis we will require the following theorem that gives guarantees on the approximation error of Algorithm \ref{alg:powerm}. Both parts of the theorem follows essentially form the same analysis. Part one upper bounds the number of iterations of the algorithm required to achieve a crude approximation to the leading eigenvalue, and part two upper bounds the number of iterations required to achieve a high-precision approximation for the leading eigenvector.

\newcommand{\TpmCrude}{T^{\textsc{PM}}_{crude}}
\newcommand{\TpmAcc}{T^{\textsc{PM}}_{acc}}

\begin{theorem}\label{thm:pm}
Let $\M$ be a positive definite matrix and denote its eigenvalues in descending order by $\lambda_1, \lambda_2,...,\lambda_d$, and let $\u_1,\u_2,...,\u_d$ denote the corresponding eigenvectors. Denote $\delta = \lambda_1 - \lambda_2$ and $\kappa = \frac{\lambda_1}{\delta}$. Fix an error tolerance $\epsilon > 0$ and failure probability $p > 0$. Define:
\begin{eqnarray*}
\TpmCrude(\epsilon,p) = \lceil{\frac{1}{\epsilon}\ln\left({\frac{18d}{p^2\epsilon}}\right)}\rceil,  \qquad
\TpmAcc(\kappa,\epsilon,p) = \lceil{\frac{\kappa}{2}\ln\left({\frac{9d}{p^2\epsilon}}\right)}\rceil .
\end{eqnarray*}
Then, with probability $1-p$ it holds that
\begin{enumerate}
\item (crude regime) $\forall t \geq \TpmCrude(\epsilon,p)$: $\w_t^{\top}\M\w_t \geq (1-\epsilon)\lambda_1$.
\item (accurate regime)  $\forall t \geq \TpmAcc(\kappa,\epsilon,p)$: $(\w_t^{\top}\u_1)^2 \geq 1-\epsilon$.
\end{enumerate}
In both cases, the success probability depends only on the random variable $(\w_0^{\top}\u_1)^2$.
\end{theorem}

A proof is given in the appendix for completeness.

\subsubsection{The Inverse Power Method and Conditioning}\label{sec:invPM}

As seen in Theorem \ref{thm:pm}, for a given matrix $\X$, the convergence rate of the \textit{Power Method} algorithm is strongly connected with the condition number $\kappa(\X)=\frac{\lambda_1(\X)}{\delta(\X)}$ which can be quite large. Consider now the following matrix	
\begin{eqnarray*}
\M^{-1} := \left({\lambda\I - \X}\right)^{-1},
\end{eqnarray*}
where $\lambda$ is a parameter.

Note that if $\X$ is positive semidefinite and $\lambda > \lambda_1(\X)$, then $\M^{-1}$ is positive definite. Furthermore, the eigenvectors of $\M^{-1}$ are the same as those of $\X$ and its eigenvalue are (in descending order) $\lambda_i(\M^{-1}) = \frac{1}{\lambda-\lambda_i(\X)}$.
Thus, if our goal is to compute the leading eigenvector of $\X$ we might as well compute the leading eigenvector of $\M^{-1}$. This is also known as the \textit{Inverse Method} \cite{MatrixCompBook}.

The following lemma shows that with a careful choice for the parameter $\lambda$, we can make the condition number $\kappa(\M^{-1})=\frac{\lambda_1(\M^{-1})}{\delta(\M^{-1})}$ to be much smaller than that of the original matrix $\X$.

\begin{lemma}[Inverse Conditioning]\label{lem:inverseCond}
Fix a scalar $a > 0$. Let $\M^{-1} = (\lambda\I-\X)^{-1}$ such that $\lambda_1(\X) + a\delta(\X) \geq \lambda > \lambda_1(\X)$. It holds that
\begin{eqnarray*}
\frac{\lambda_1(\M^{-1})}{\delta(\M^{-1})} \leq 1+ a .
\end{eqnarray*}
\end{lemma}

\begin{proof}
We denote by $\lambda_1,\lambda_2,\delta$ the values $\lambda_1(\X), \lambda_2(\X)$ and $\delta(\X)$ respectively.

It holds that
\begin{eqnarray*}
\lambda_1(\M^{-1}) = \frac{1}{\lambda- \lambda_1} , \quad \lambda_2(\M^{-1}) = \frac{1}{\lambda - \lambda_2} .
\end{eqnarray*}

Thus we have that
\begin{eqnarray*}
\frac{\lambda_1(\M^{-1})}{\delta(\M^{-1})}  &=& \frac{1}{\lambda - \lambda_1}\left({\frac{1}{\lambda - \lambda_1} - \frac{1}{\lambda - \lambda_2}}\right)^{-1} \\
&=& 
\frac{1}{\lambda - \lambda_1}\left({\frac{1}{\lambda - \lambda_1} - \frac{1}{\lambda - \lambda_1+\delta}}\right)^{-1} = \frac{1}{\lambda - \lambda_1}\left({\frac{\delta}{(\lambda-\lambda_1)(\lambda-\lambda_1+\delta)}}\right)^{-1} \\
&=& \frac{\lambda-\lambda_1+ \delta}{\delta} \leq 1 + \frac{a\delta}{\delta} = 1 + a,
\end{eqnarray*}
and the lemma follows.
\end{proof}

Of course, a problem with the above suggested approach is that we don't know how to set the parameter $\lambda$ to be close enough to $\lambda_1(\X)$.  The following simple lemma shows that by approximating the largest eignevalue of $(\lambda\I-\X)^{-1}$, we can derive bounds on the suboptimality gap $\lambda-\lambda_1(\X)$, which in turn can be used to better tune $\lambda$.

\begin{lemma}\label{lem:invapprox}
Fix a positive semidefinite matrix $\X$ and denote the largest eigenvalue by $\lambda_1$. Let $\lambda > \lambda_1$ and consider a unit vector $\w$ such that
\begin{eqnarray*}
\w^{\top}\left({\lambda\I - \X}\right)^{-1}\w \geq (1-\epsilon)\lambda_1\left({\left({\lambda\I - \X}\right)^{-1}}\right) ,
\end{eqnarray*}
for some $\epsilon \in (0,1]$.

Denote $\Delta = \frac{1-\epsilon}{\w^{\top}\left({\lambda\I - \X}\right)^{-1}\w}$. 
Then it holds that 
\begin{eqnarray*}
(1-\epsilon)(\lambda-\lambda_1) \leq \Delta \leq \lambda-\lambda_1 .
\end{eqnarray*}
\end{lemma}

\begin{proof}

According to our assumption on $\w$ it holds that
\begin{eqnarray*}
\lambda_1\left({(\lambda\I - \X)^{-1}}\right) \geq \w^{\top}\left({\lambda\I - \X}\right)^{-1}\w \geq (1-\epsilon)\lambda_1\left({\left({\lambda\I - \X}\right)^{-1}}\right) .
\end{eqnarray*}
Thus by our definition of $\Delta$ and the fact that $\lambda_1\left({\lambda\I - \X)^{-1}}\right) = \frac{1}{\lambda-\lambda_1}$, it holds that
\begin{eqnarray*}
(1-\epsilon)(\lambda-\lambda_1) \leq \Delta \leq \lambda-\lambda_1,
\end{eqnarray*}

and the lemma follows.

\end{proof}

Thus, if we set for instance $\epsilon=1/2$, and then define $\lambda' = \lambda - \Delta$, then by Lemma \ref{lem:invapprox} we have that $\lambda' - \lambda_1 \leq \frac{\lambda-\lambda_1}{2}$.

\subsection{Matrix Inversion via Convex Optimization}\label{sec:matrixInv}

\subsubsection{Smoothness and strong convexity of continuous functions}
\begin{definition}[smooth function]\label{def:smoothfunc}
We say that a function $f:\mathbb{R}^d\rightarrow\mathbb{R}$ is $\beta$ smooth if for all $\x,\y\in\mathbb{R}^d$ it holds that
\begin{eqnarray*}
\Vert{\nabla{}f(\x)-\nabla{}f(\y)}\Vert \leq \beta\Vert{\x-\y}\Vert .
\end{eqnarray*}
\end{definition}

\begin{definition}[strongly convex function]\label{def:strongconvexfunc}
We say that a function $f:\mathbb{R}^d\rightarrow\mathbb{R}$ is $\alpha$-strongly convex if for all $\x,\y\in\mathbb{R}^d$ it holds that
 \begin{eqnarray*}
f(\y) \geq f(\x) + \nabla{}f(\x)^{\top}(\y-\x) + \frac{\alpha}{2}\Vert{\x-\y}\Vert^2 .
\end{eqnarray*}
\end{definition}

The above definition combined with first order optimality conditions imply that for an $\alpha$-strongly convex function $f$, if $\x^*=\arg\min_{\x\in\mathbb{R}^d}f(\x)$, then for any $\x\in\mathbb{R}^d$ it holds that
\begin{eqnarray}\label{ie:strongconvex}
f(\x)-f(\x^*) \geq \frac{\alpha}{2}\Vert{\x-\x^*}\Vert^2 .
\end{eqnarray}

\begin{lemma}[smoothness and strong convexity of quadratic functions]\label{lem:quadFuncProp}
Consider the function 
\begin{eqnarray*}
f(\x) = \frac{1}{2}\x^{\top}\M\x + \b^{\top}\x, 
\end{eqnarray*}
where $\M\in\mathbb{R}^{d\times d}$ is symmetric and $\b\in\mathbb{R}^d$. If $\M \succeq{0}$ then $f(\x)$ is $\lambda_1(\M)$-smooth and $\lambda_d(\M)$-strongly convex, where $\lambda_1, \lambda_d$ denote the largest and smallest eigenvalues of $\M$ respectively.Otherwise, $f(\x)$ is $\Vert{\M}\Vert$-smooth.
\end{lemma}

\subsubsection{Matrix inversion via convex optimization}

In order to apply the Inverse Method discussed in the previous subsection, we need to apply Power Method steps (Algorithm \ref{alg:powerm}) to the matrix $(\lambda\I-\X)^{-1}$. Denote $\M = \lambda\I-\X$. Thus on each iteration of the Inverse Method we need to compute a matrix-vector product of the form $\M^{-1}\w$. This requires in general to solve a linear system of equations. However, it could be also approximated arbitrarily well using convex optimization. Consider the following optimization problem:
\begin{eqnarray}\label{eq:matrixInvProb}
\min_{\z\in\mathbb{R}^d}\{F(\z) := \frac{1}{2}\z^{\top}\M\z - \w^{\top}\z \} .
\end{eqnarray}

By the first order optimality condition, we have that an optimal solution for Problem \eqref{eq:matrixInvProb} - $\z^*$ satisfies that $\nabla{}F(\z^*) = \M\z^* - \w = 0$, meaning,
\begin{eqnarray*}
\z^* = \M^{-1}\w .
\end{eqnarray*}

Note that under the assumption that $\lambda > \lambda(\X)$ (as stated Subsection \ref{sec:invPM}) it holds that $\M$ is positive definite and hence invertible.

Most importantly, note that under this assumption on $\lambda$ it further follows from Lemma \ref{lem:quadFuncProp} that $F(\z)$ is $\lambda_d(\M) = (\lambda-\lambda_1(\X))$-strongly convex and $\lambda_1(\M) = (\lambda - \lambda_d(\X))$-smooth and thus could be minimized very efficiently via algorithms for convex minimization.

Since by using algorithms for convex minimization we can only find an approximated-minimizer of $F(\z)$, we must discuss the effect of the approximation error on the convergence of the proposed algorithms. As it will turn out, the approximation error that we will care about it the distance $\Vert{\z-\z^*}\Vert$ where $\z$ is an approximated minimizer of $F$. The following lemma, which follows directly from the strong convexity of $F$ and Eq. \eqref{ie:strongconvex}, ties between the approximation error of a point $\z$ with respect to the function $F$ and the distance to the optimal solution $\z^*$.

\begin{lemma}
Given a positive semidefinite matrix $\X$, a vector $\w$, a scalar $\lambda$ such that $\lambda > \lambda_1(\X)$ and an error tolerance $\epsilon$, let $\M = \lambda\I-\X$, and denote by $\z^*$ the minimizer of $F(\z)$ - as defined in Eq. \eqref{eq:matrixInvProb}. Then, for any $\z$ it holds that
\begin{eqnarray*}
\Vert{\z-\z^*}\Vert \leq \sqrt{\frac{2(F(\z)-F(\z^*))}{\lambda-\lambda_1(\X)}} .
\end{eqnarray*}
\end{lemma}

\subsubsection{Fast stochastic gradient methods for smooth and strongly convex optimization}\label{sec:stochasticOptPrem}

In this subsection we briefly survey recent developments in stochastic optimization algorithms for convex optimization which we leverage in our analysis in order to get the fast rates for PCA.

Consider an optimization problem of the following form
\begin{equation}\label{eq:convexprob}
\min_{\z}\{F(\z): = \frac{1}{n}\sum_{i=1}^nf_i(\z)\} \tag{P}
\end{equation}
where we assume that each $f_i$ is convex and $\beta$-smooth and that the sum $F(\z)$ is $\sigma$-strongly convex.


We are going to show that the PCA problem could be reduced to solving a series of convex optimization problems that takes the form of Problem \eqref{eq:convexprob} (this is actually not a precise statement since in our case each function $f_i$ won't be convex on its own and we will need to address this issue). Thus, we are interested in fast algorithms for solving Problem \eqref{eq:convexprob}. 

The standard gradient descent method can solve Problem \eqref{eq:convexprob} to $\epsilon$ precision in  $O\left({(\beta/\sigma)\log(1/\epsilon)}\right)$ iterations were each iteration requires to compute the gradient of $F(\z)$. Thus the overall time becomes  $\tilde{O}\left({\frac{\beta}{\sigma}T_G}\right)$ where we denote by $T_G$ the time to evaluate the gradient direction of $F$.
of any of the functions $f_i$.
The dependence on the condition number $\frac{\beta}{\sigma}$ could be dramatically improved without increasing significantly the per-iteration complexity, by using Nesterov's accelerated method that requires $O\left({\sqrt{\beta/\sigma}\log(1/\epsilon)}\right)$ iterations, and overall time of $\tilde{O}\left({\sqrt{\beta/\sigma}T_G}\right)$ 

However, both methods could be quite computationally expensive when both $\frac{\beta}{\sigma}$ and $T_g$ are large.

Another alternative is to use stochastic gradient descent, which on each iteration $t$, performs a gradient improvement step based on a single function $f_{i_t}$ where $i_t$ is chosen uniformly at random from $[n]$. This single random gradient serves as an estimator for the full gradient. The benefit of this method is that each iteration is extremely cheap - only requires to evaluate the gradient of a single function. However the convergence rate of this method is roughly $1/(\sigma\epsilon)$ which is ineffective when $\epsilon$ is very small. The intuitive reason for the slow convergence is the large variance of the gradient estimator.

Each of the methods mentioned above, the deterministic gradient descent, and the stochastic one has its own caveats. The deterministic gradient method does not consider the special structure of Problem \eqref{eq:convexprob} which is given by a sum of functions, and on the other hand, the stochastic gradient method does not exploit the fact that the sum of functions is finite. 

Recently, a new family of stochastic gradient descent-based methods was devised, which is tailored to tackling Problem \eqref{eq:convexprob} \cite{RouxSB12, ShwartzZhang, Zhang13, Zhang13b, KonecnyR13, Harchaoui15, Kakade15}. Roughly speaking, these methods apply cheap stochastic gradient descent update steps, but use the fact the the objective function is given in the form a finite sum, to construct a gradient estimator with reduced variance. 

For instance, the \textsc{SVRG} algorithm presented in \cite{Zhang13}, requires $O(\log(1/\epsilon))$ iterations to reach $\epsilon$ accuracy, where each iteration requires computing a single full gradient of $F(\z)$ and roughly $O(\beta/\sigma)$ cheap stochastic gradient updates. Thus the total running time becomes $O\left({\left({\frac{\beta}{\sigma}T_g + T_G}\right)\log(1/\epsilon)}\right)$, where $T_g$ denotes the worst case time to evaluate the gradient of a single function $f_i$.

The following theorem summarizes the application of \textsc{SVRG} to solving Problem \eqref{eq:convexprob}. For details see \cite{Zhang13}.

\begin{theorem}[Convergnec of SVRG]\label{thm:svrg}
Fix $\epsilon > 0$. Assume each $f_i(\z)$ is $\beta$-smooth and $F(\z)$ is $\sigma$-strongly convex. Then the \textsc{SVRG} Algorithm detailed in \cite{Zhang13} finds in total time $$O\left({\left({\frac{\beta}{\sigma}T_g + T_G}\right)\log(1/\epsilon)}\right)$$ a vector $\hat{\z}\in\mathbb{R}^d$ such that
\begin{eqnarray*}
\E[F(\hat{\z})] - \min_{\z\in\mathbb{R}^d}F(\z) \leq \epsilon .
\end{eqnarray*}
\end{theorem}

In recent works \cite{Kakade15, Harchaoui15}, it was demonstrated how methods such as the SVRG algorithm could be further accelerated by improving the dependence on the condition number $\beta/\sigma$.




\section{The Basic Approach: a Shrinking Inverse Power Algorithm}\label{sec:PowerMethodBasedAlg} 

The goal of this section is to present a Power Method-based algorithm that requires to compute an overall \textbf{poly-logarithmic} number of matrix-vector products in order to find an $\epsilon$ approximation for the leading eigenvector of a given positive semidefinite matrix $\X$.

\begin{algorithm}[h]
\caption{\textsc{Shrinking Inverse Power Method}}
\label{alg:shrinkInvPM}
\begin{algorithmic}[1]
\STATE Input: matrix $\X\in\mathbb{R}^{n\times n}$ such that $\X\succeq 0$, $\lambda_1(\X) \leq 1$, an estimate $\hat{\delta}$ for $\delta(\X)$,  accuracy parameter $\epsilon\in(0,1)$
\STATE $\lambda_{(0)} \gets 1+\hat{\delta}$
\STATE $s \gets 0$
\REPEAT
\STATE $s \gets s+1$
\STATE Let $\M_s = (\lambda_{(s-1)}\I-\X)$
\STATE Apply the \textsc{Power Method} (Algorithm \ref{alg:powerm}) to the matrix $\M_s^{-1}$ to find a unit vector
$\w_s$ such that 
$$\w_s^{\top}\M_s^{-1}\w_s \geq \frac{1}{2}\lambda_1(\M_s^{-1})$$
\STATE $\Delta_s \gets \frac{1}{2}\cdot\frac{1}{ \w_s^{\top}\M_s^{-1}\w_{s}}$
\STATE $\lambda_{(s)} \gets \lambda_{(s-1)} - \frac{\Delta_s}{2}$
\UNTIL{$\Delta_s \leq \hat{\delta}$}
\STATE $\lambda_{(f)} \gets \lambda_{(s)}$
\STATE Let $\M_f = (\lambda_{(f)}\I-\X)$
\STATE Apply the \textsc{Power Method} (Algorithm \ref{alg:powerm}) to the matrix  $\M_f^{-1}$ to find a unit vector $\w_f$ such that 
$$(\w_f^{\top}\u)^2 \geq 1-\epsilon$$
\RETURN $\w_f$
\end{algorithmic}
\end{algorithm}

We prove the following theorem. 
\begin{theorem}\label{thm:basicAlgConv}
Assume $\hat{\delta}$ satisfies that $\hat{\delta} \in\left[{\frac{\delta}{2}, 2\delta}\right]$. There exists an implementation for Algorithm \ref{alg:shrinkInvPM} that requires computing at most
\begin{eqnarray*}
O\left({\log(d/p)\log(\delta^{-1}) + \log\left({\frac{d}{p\epsilon}}\right)}\right)  = \tilde{O}(1)
\end{eqnarray*}
matrix-vector products of the form $\M^{-1}\w$, where $\M$ is one of the matrices computed during the run of the algorithm ($\M_s$ or $\M_f$) and $\w$ is some vector,
such that with probability at least $1-p$ it holds that the output of the algorithm, the vector $\w_f$, satisfies:
\begin{eqnarray*}
(\w_f^{\top}\u)^2 \geq 1-\epsilon .
\end{eqnarray*}
\end{theorem}

In order to prove Theorem \ref{thm:basicAlgConv} we need a few simple Lemmas.

First, it is important that throughout the run of Algorithm \ref{alg:shrinkInvPM}, the matrices $\M_s$ and the matrix $\M_f$ will be positive definite (and as a results so are their inverses). The following lemma shows that this is indeed the case.

\begin{lemma}\label{lem:postivieGap}
For all $s\geq 0$ it holds that $\lambda_{(s)} > \lambda_1$.
\end{lemma}

\begin{proof}
The proof is by a simple induction. The claim clearly holds for $s=0$ since by our assumption $\lambda_1 \leq 1$. Suppose now that the claim holds for some iteration $s$.
According to Lemma \ref{lem:invapprox} it holds that the value $\Delta_{s+1}$ computed on iteration $s+1$ satisfies that
\begin{eqnarray*}
\Delta_{s+1} \leq \lambda_{(s)} - \lambda_1 .
\end{eqnarray*}
Hence, according to the algorithm, it holds that
\begin{eqnarray*}
\lambda_{(s+1)} = \lambda_{(s)} - \frac{\Delta_{s+1}}{2} \geq  \lambda_{(s)} - \frac{\lambda_{(s)} - \lambda_1}{2}
= \frac{\lambda_{(s)} + \lambda_1}{2} > \lambda_1, 
\end{eqnarray*}
where the last inequality follows from the induction hypothesis.
\end{proof}

The following lemma bounds the number of iterations of the loop in Algorithm \ref{alg:shrinkInvPM}.
\begin{lemma}\label{lem:loopBound}	
The \textit{repeat-until} loop is executed at most $O(\log(\hat{\delta}^{-1}))$ times.
\end{lemma}
\begin{proof}
Fix an iteration $s$ of the loop. By applying Lemma \ref{lem:invapprox} with respect to the unit vector $\w_s$ we have that
\begin{eqnarray*}
\Delta_s \geq \frac{1}{2}(\lambda_{(s-1)}-\lambda_1).
\end{eqnarray*}
By the update rule of the algorithm it follows that
\begin{eqnarray}\label{eq:gapDec}
\lambda_{(s)} - \lambda_1 &=& \lambda_{(s-1)} - \frac{\Delta_s}{2} - \lambda_1 
\leq (\lambda_{(s-1)} - \lambda_1) - \frac{1}{4}(\lambda_{(s-1)}-\lambda_1) \nonumber \\
&=& \frac{3}{4}(\lambda_{(s-1)}-\lambda_1) .
\end{eqnarray}

Thus, after at most $T = \lceil{\log_{3/4}\left({\frac{\hat{\delta}}{\lambda_{(0)}-\lambda_1}}\right)}\rceil = O(\log(\hat{\delta}^{-1}))$ (using our choice of $\lambda_{(0)}$) iterations, we arrive at a value $\lambda_{(T)}$ which satisfies $\lambda_{(T)} - \lambda_1 \leq \hat{\delta}$. By Lemma \ref{lem:invapprox} it follows that in the following iteration it will hold that $\Delta_{T+1} \leq \hat{\delta}$ and the loop will end. Hence, the overall number of iterations is at most $T+1 = O(\log(\hat{\delta}^{-1}))$.
\end{proof}

Finally, the following lemma gives approximation guarantees on the estimate $\lambda_{(f)}$.

\begin{lemma}\label{lem:finalGapBound}
Suppose that all executions of the \textit{Power Method} in Algorithm \ref{alg:shrinkInvPM} are successful. Then
it holds that 
\begin{eqnarray}
\lambda_1 + \frac{3\hat{\delta}}{2} \geq \lambda_{(f)} \geq \lambda_1 + \frac{\hat{\delta}}{4}.
\end{eqnarray}

\end{lemma}
\begin{proof}
Denote by $s_f$ the last iteration of the loop in Algorithm \ref{alg:shrinkInvPM}, and note that using this notation we have that $\lambda_{(f)} = \lambda_{(s_f)}$ and that $\Delta_{s_f} \leq \hat{\delta}$. Using Lemma \ref{lem:invapprox}, we thus have that
\begin{eqnarray*}
\lambda_{(f)} - \lambda_1 = \lambda_{(s_f-1)} - \frac{\Delta_{s_f}}{2} - \lambda_1 \leq 2\Delta_{s_f} - \frac{\Delta_{s_f}}{2} = \frac{3}{2}\Delta_{s_f} \leq \frac{3}{2}\hat{\delta},
\end{eqnarray*}
which gives the first part of the lemma.

For the second part, using Lemma \ref{lem:invapprox} again, we have that
\begin{eqnarray}\label{eq:2}
\lambda_{(f)} - \lambda_1 &=& \lambda_{(s_f-1)} - \frac{\Delta_{s_f}}{2} - \lambda_1 \geq \lambda_{(s_f-1)} - \lambda_1 - \frac{1}{2}(\lambda_{(s_f-1)}-\lambda_1) \nonumber \\
&=& \frac{1}{2}(\lambda_{(s_f-1)}-\lambda_1) .
\end{eqnarray}

In case $s_f = 1$, then by our choice of $\lambda_{(0)}$ we have that $\lambda_{(0)}-\lambda_1 \geq \hat{\delta}$, and the lemma follows. 
Otherwise, by unfolding Eq. \eqref{eq:2} one more time, we have that
\begin{eqnarray*}
\lambda_{(f)} - \lambda_1 \geq  \frac{1}{4}(\lambda_{(s_f-2)}-\lambda_1) \geq \frac{\Delta_{(s_f-1)}}{4} > \frac{\hat{\delta}}{4},
\end{eqnarray*}
where the second inequality follows from Lemma \ref{lem:invapprox} and the last inequality follows from the stopping condition of the loop.
\end{proof}

We can now prove Theorem \ref{thm:basicAlgConv}.

\begin{proof}

First a note regarding the success probability of the invocations of the \textit{Power Method} algorithm  in Algorithm \ref{alg:shrinkInvPM}: since, as stated in Theorem \ref{thm:pm}, the success of the PM algorithm depends only on the magnitude of $(\w_0^{\top}\u)^2$, and all matrices $\M_s,\M_f$ in Algorithm \ref{alg:shrinkInvPM} have the same leading eigenvector, as long as all invocations use the same random initial vector and number of steps that guarantees success with probability $1-p$, they all succeed together with probability at least $1-p$.

Let us now assume that all executions of the \textit{Power Method} algorithm in Algorithm \ref{alg:powerm} were successful.

According to Lemma \ref{lem:loopBound}, the loop is executed at most $O(\log(\hat{\delta}^{-1})) = O(\log(\delta^{-1}))$ times (following our assumption on $\hat{\delta}$).
Each iteration $s$ of the loop requires to invoke the \textit{Power Method} to approximate $\lambda_1(\M)$ up to a factor of $1/2$ which according to Theorem \ref{thm:pm}, requires computing $\TpmCrude(1/2,p) = O(\log(d/p))$ matrix-vector products, in order to succeed with probability at least $1-p$. 
Thus the overall number of matrix-vector products computed during the loop is $O\left({\log(d/p)\log(\delta^{-1})}\right)$.

According to lemma \ref{lem:finalGapBound} it holds that $\lambda_{(f)}-\lambda_1 \leq \frac{3}{2}\hat{\delta} \leq 3\delta$. Thus, according to Lemma \ref{lem:inverseCond}, we have that $$\kappa(\M_f^{-1}) = \frac{\lambda_1(\M_f^{-1})}{\delta(\M_f^{-1})} \leq 4 = O(1) .$$

Thus, in the final invocation of the PM algorithm, it requires at most $\TpmAcc(4, \epsilon, p) = O\left({\log\left({\frac{d}{p\epsilon}}\right)}\right)$ matrix-vector products to compute $\w_f$ as desired, with probability at least $1-p$.

Thus the overall number of matrix-vector products is
\begin{eqnarray*}
O\left({\log(d/p)\log(\delta^{-1}) + \log\left({\frac{d}{p\epsilon}}\right)}\right) .
\end{eqnarray*}

\end{proof}

\section{A Convex Optimization-based Eigenvector Algorithm}\label{sec:convexEV}

In this section we present our algorithm for approximating the largest eigenvector of a given matrix based on convex optimization.
The algorithm is based on Algorithm \ref{alg:shrinkInvPM} from the previous section, but replaces explicit computation of products between vectors and inverted matrices, with solving convex optimization problems, as detailed in Subsection \ref{sec:matrixInv}.

Towards this end, we assume that we are given access to an algorithm - $\mA$ for solving problems of the following structure:
\begin{eqnarray}\label{eq:pcaOptProblem} 
\min_{\z\in\mathbb{R}^d}\lbrace{F_{\w,\lambda}(\z) := \frac{1}{2}\z^{\top}(\lambda\I-\X)\z - \w^{\top}\z}\rbrace,
\end{eqnarray}
where $\X$ is positive definite, $\lambda > \lambda_1(\X)$ and $\w$ is some vector. Note that under these conditions, the function $F_{\w,\lambda}(\z)$ is strongly convex. Note also that the minimizer of $F_{\w,\lambda}(\z)$ - $\z^*$ is given by $\z^*=(\lambda\I-\X)^{-1}\w$, and thus solving Problem \eqref{eq:pcaOptProblem} is equivalent to computing a product between a vector and an inverse matrix.

There are a few issues with our approach that require delicate care however: 1) we need to pay close attention that the convex optimization problems are well-conditioned and  2) since now we use a numerical procedure to compute the matrix-vector products, we have approximation errors that we need to consider.


\begin{algorithm}
\caption{\textsc{Leading Eigenvector via Convex Optimization}}
\label{alg:convexEV}
\begin{algorithmic}[1]
\STATE Input: matrix $\X\in\mathbb{R}^{n\times n}$ such that $\X\succeq 0$, $\lambda_1(\X) \leq 1$, an estimate $\hat{\delta}$ for $\delta(\X)$,  accuracy parameter $\epsilon\in(0,1)$, failure probability parameters $p$
\STATE $\lambda_{(0)} \gets 1+ \hat{\delta}$
\STATE Set: $m_1 \gets \TpmCrude(1/8, p), \quad m_2 \gets \TpmAcc(3, \epsilon/2, p)$
\STATE Set: $\tilde{\epsilon} \gets \min\{\frac{1}{16}\left({\frac{\hat{\delta}}{8}}\right)^{m_1+1}, \frac{\epsilon}{4}\left({\frac{\hat{\delta}}{8}}\right)^{m_2+1}\} $
\STATE Let $\hat{\w}_0$ be a random unit vector
\STATE $s \gets 0$
\REPEAT
\STATE $s \gets s+1$
\STATE Let $\M_s = (\lambda_{(s-1)}\I-\X)$
\FOR{$t=1...m_1$}
\STATE Apply Algorithm $\mA$ to find a vector $\hat{\w}_t$ such that $\Vert{\hat{\w}_t - \M_s^{-1}\hat{\w}_{t-1}}\Vert \leq \tilde{\epsilon}$
\ENDFOR
\STATE $\w_s \gets \frac{\hat{\w}_{m_1}}{\Vert{\hat{\w}_{m_1}}\Vert}$
\STATE Apply Algorithm $\mA$ to find a vector $\v_s$ such that $\Vert{\v_s - \M_s^{-1}\w_s}\Vert \leq \tilde{\epsilon}$
\STATE $\Delta_s \gets \frac{1}{2}\cdot\frac{1}{ \w_s^{\top}\v_s - \tilde{\epsilon}}$
\STATE $\lambda_{(s)} \gets \lambda_{(s-1)} - \frac{\Delta_s}{2}$
\UNTIL{$\Delta_s \leq \hat{\delta}$}
\STATE $\lambda_{(f)} \gets \lambda_{(s)}$
\STATE Let $\M_f = (\lambda_{(f)}\I-\X)$
\FOR{$t=1...m_2$}
\STATE Apply Algorithm $\mA$ to find a vector $\hat{\w}_t$ such that $\Vert{\hat{\w}_t - \M_f^{-1}\hat{\w}_{t-1}}\Vert \leq \tilde{\epsilon}$
\ENDFOR
\RETURN $\w_f \gets \frac{\hat{\w}_{m_2}}{\Vert{\hat{\w}_{m_2}}\Vert}$
\end{algorithmic}
\end{algorithm}

\subsection{Power Method with inaccurate updates}

In this section we analyze the potential convergence of the Power Method with inaccurate matrix-vector products.

\begin{lemma}[Power Method with inaccurate matrix-vector products]\label{lem:approxProduct}
Let $\M$ be a positive definite matrix with largest eigenvalue $\lambda_1$ and smallest eigenvalue $\lambda_d$. Fix an accuracy parameter $\epsilon > 0$. Let $\w$ be an arbitrary unit vector and consider the following sequences  of vectors $\{\w_t^*\}_{t=0}^{\infty}, \{\hat{\w}_t^*\}_{t=0}^{\infty}$ and $\{\w_t\}_{t=0}^{\infty}, \{\hat{\w}_t\}_{t=0}^{\infty}$ defined as follows: 
\begin{eqnarray*}
&&\w_0^* = \hat{\w}_0^* = \w, \quad \forall t\geq 1: \, \hat{\w}_t^* \gets \M\hat{\w}_{t-1}^* , \quad \w_t^* \gets \frac{\hat{\w}_t^*}{\Vert{\hat{\w}_t^*}\Vert},\\
&&\w_0 = \hat{\w}_0 = \w , \quad \forall t\geq 1: \,  \textrm{$\hat{\w}_t$ satisfies that }  \Vert{\hat{\w}_t - \M\hat{\w}_{t-1}}\Vert \leq \epsilon, \quad \w_t \gets \frac{\hat{\w}_t}{\Vert{\hat{\w}_t}\Vert}.
\end{eqnarray*}
Denote:
\begin{eqnarray*}
\Gamma(\M,t) :=  \frac{2}{\lambda_d^t}\left\{ \begin{array}{ll}
         t & \mbox{if $\lambda_1 = 1$}\\
        \frac{\lambda_1^t-1}{\lambda_1-1} & \mbox{if $\lambda_1 \neq 1$}\end{array} \right. ,
        \qquad
\hat{\Gamma}(\M,t) := \left\{ \begin{array}{ll}
         t & \mbox{if $\lambda_1 = 1$}\\
        \frac{\lambda_1^t-1}{\lambda_1-1} & \mbox{if $\lambda_1 \neq 1$} \end{array} \right.   .     
\end{eqnarray*}
Then it holds that for all $t\geq 0$, 
\begin{eqnarray*}
\Vert{\hat{\w}_t - \hat{\w}_t^*}\Vert \leq \epsilon\cdot\hat{\Gamma}(\M,t)
\end{eqnarray*}
and
\begin{eqnarray*}
\Vert{\w_t - \w_t^*}\Vert \leq \epsilon\cdot\Gamma(\M,t)
\end{eqnarray*}
\end{lemma}
\begin{proof}
First, observe that:
\begin{eqnarray*}
\Vert{\hat{\w}_{t+1} - \hat{\w}_{t+1}^*}\Vert &=& \Vert{\hat{\w}_{t+1} - \M\hat{\w}_t + \M\hat{\w}_t -\hat{\w}_{t+1}^*}\Vert
\leq \Vert{\hat{\w}_{t+1} - \M\hat{\w}_t}\Vert + \Vert{\M\hat{\w}_t - \M\hat{\w}_t^*}\Vert \\
& \leq & \epsilon + \lambda_1\cdot\Vert{\hat{\w}_t - \hat{\w}_t^*}\Vert .
\end{eqnarray*}

In case $\lambda_1 = 1$, we clearly have that
\begin{eqnarray}\label{eq:6}
\Vert{\hat{\w}_{t+1} - \hat{\w}_{t+1}^*}\Vert \leq (t+1)\epsilon + \Vert{\hat{\w}_{0} - \hat{\w}_{0}^*}\Vert = (t+1)\epsilon .
\end{eqnarray}
Otherwise, in case $\lambda_1 \neq 1$, by a simple algebraic manipulation we have that
\begin{eqnarray*}
\Vert{\hat{\w}_{t+1} - \hat{\w}_{t+1}^*}\Vert + \frac{\epsilon}{\lambda_1-1}
& \leq &  \lambda_1\cdot\Vert{\hat{\w}_t - \hat{\w}_t^*}\Vert + \frac{\epsilon\lambda_1}{\lambda_1-1}
= \lambda_1\left({\Vert{\hat{\w}_t - \hat{\w}_t^*}\Vert + \frac{\epsilon}{\lambda_1-1}}\right).
\end{eqnarray*}

It thus follows that for all $t\geq 0$,
\begin{eqnarray}\label{eq:1}
\Vert{\hat{\w}_{t+1} - \hat{\w}_{t+1}^*}\Vert \leq \lambda_1^{t+1}\left({\Vert{\hat{\w}_0 - \hat{\w}_0^*}\Vert + \frac{\epsilon}{\lambda_1-1}}\right) - \frac{\epsilon}{\lambda_1-1} 
= \frac{\epsilon}{\lambda_1-1}(\lambda_1^{t+1} - 1).
\end{eqnarray}

We now have that for all $t\geq 1$ it holds that
\begin{eqnarray*}
\Vert{\w_t - \w_t^*}\Vert &=& \Vert{\frac{\hat{\w}_t}{\Vert{\hat{\w}_t}\Vert} - \frac{\hat{\w}_t^*}{\Vert{\hat{\w}_t^*}\Vert}}\Vert
=\Vert{\frac{\hat{\w}_t}{\Vert{\hat{\w}_t}\Vert} - \frac{\hat{\w}_t}{\Vert{\hat{\w}_t^*}\Vert} + \frac{\hat{\w}_t}{\Vert{\hat{\w}_t^*}\Vert} - \frac{\hat{\w}_t^*}{\Vert{\hat{\w}_t^*}\Vert}}\Vert \\
& \leq & \Vert{\frac{\hat{\w}_t}{\Vert{\hat{\w}_t^*}\Vert} - \frac{\hat{\w}_t^*}{\Vert{\hat{\w}_t^*}\Vert}}\Vert
+ \Vert{\frac{\hat{\w}_t}{\Vert{\hat{\w}_t}\Vert} - \frac{\hat{\w}_t}{\Vert{\hat{\w}_t^*}\Vert}}\Vert \\
&= & \frac{1}{\Vert{\hat{\w}_t^*}\Vert}\cdot\Vert{\hat{\w}_t - \hat{\w}_t^*}\Vert + \Vert{\hat{\w}_t}\Vert \cdot \vert{\frac{1}{\Vert{\hat{\w}_t}\Vert} - \frac{1}{\Vert{\hat{\w}_t^*}\Vert}}\vert \\
& = &\frac{1}{\Vert{\hat{\w}_t^*}\Vert}\cdot\Vert{\hat{\w}_t - \hat{\w}_t^*}\Vert + \frac{\vert{\Vert{\hat{\w}_t^*}\Vert - \Vert{\hat{\w}_t}\Vert}\vert}{\Vert{\hat{\w}_t^*}\Vert} \leq \frac{2\Vert{\hat{\w}_t - \hat{\w}_t^*}\Vert}{\Vert{\hat{\w}_t^*}\Vert},
\end{eqnarray*}
where the last inequality follows from the triangle inequality.

Now, since $\hat{\w}_t^* = \M^t\w$ and $\M \succeq \lambda_{d}\I$, it follows that 
\begin{eqnarray*}
\Vert{\hat{\w}_t^*}\Vert \geq \lambda_d^t\cdot\Vert{\w}\Vert = \lambda_d^t .
\end{eqnarray*}

Combining this with Eq. \eqref{eq:1} and \eqref{eq:6}, we have that for all $t\geq 1$,

\begin{eqnarray*}
\Vert{\w_t - \w_t^*}\Vert \leq \frac{2\epsilon}{\lambda_d^t}
 \left\{ \begin{array}{ll}
         t & \mbox{if $\lambda_1 = 1$};\\
        \frac{\lambda_1^t-1}{\lambda_1-1} & \mbox{if $\lambda_1 \neq 1$}.\end{array} \right.
\end{eqnarray*}

Thus the lemma follows.
\end{proof}

The following corollary is a consequence of the convergence result for the Power Method (Theorem \ref{thm:pm}) and  Lemma \ref{lem:approxProduct}.

\begin{theorem}[Convergence of the Power Method with inaccurate updates]\label{thm:approxPM}
Fix $\epsilon > 0$ and $p > 0$. Let $\M$ be a positive definite matrix with largest eigevalue $\lambda_1$ and eigengap $\delta > 0$. Consider a sequence of vectors $\lbrace{\w_t}\rbrace_{t=0}^{\tau}$ where $\w_0$ is a random unit vector, and for all $t\in[\tau]$ it holds that $\Vert{\w_t - \M\w_{t-1}}\Vert\leq \tilde{\epsilon} := \frac{\epsilon}{4\Gamma(\M,\tau)}$ (see Lemma \ref{lem:approxProduct} for definition of $\Gamma(\M,\tau)$). Then the following two guarantees hold:
\begin{enumerate}
\item  If $\tau \geq \TpmCrude(\epsilon/2,p)$ then with probability at least $1-p$ it holds that
\begin{eqnarray*}
\w_{\tau}^{\top}\M\w_{\tau} \geq (1-\epsilon)\lambda_1 .
\end{eqnarray*}

\item  If $\tau \geq \TpmAcc(\kappa(\M),\epsilon/2,p)$then with probability at least $1-p$ it holds that
\begin{eqnarray*}
(\w_{\tau}^{\top}\u_1)^2 \geq 1-\epsilon,
\end{eqnarray*}
where $\u_1$ is the largest eigenvector of $\M$.
\end{enumerate}
In both cases, the probability of success depends only on the random variable $(\w_0^{\top}\u_1)^2$.
\end{theorem}

\begin{proof}
Let $\{\w_t^*\}_{t=0}^{\infty}$ be a sequence of unit vectors as defined in Lemma \ref{lem:approxProduct}.
For the first item, note that
\begin{eqnarray*}
\w_{\tau}^{\top}\M\w_{\tau} &=& \w_{\tau}^{*\top}\M\w_{\tau}^* + \left({\w_{\tau}^{\top}\M\w_{\tau} - \w_{\tau}^{*\top}\M\w_{\tau}^*}\right) .
\end{eqnarray*}

Since $\M$ is positive definite, the function $f(\w) = \w^{\top}\M\w$ is convex and we have that
\begin{eqnarray*}
\vert{\w_{\tau}^{\top}\M\w_{\tau} - \w_{\tau}^{*\top}\M\w_{\tau}^*}\vert
&\leq & 2\max\{(\M\w_{\tau})^{\top}(\w_{\tau} - \w_{\tau}^*),\, (\M\w_{\tau}^*)^{\top}(\w_{\tau}^*-\w_{\tau})\} \\
&\leq & 2\lambda_1(\M)\cdot\Vert{\w_{\tau} - \w_{\tau}^*}\Vert .
\end{eqnarray*}

Thus we have that
\begin{eqnarray*}
\w_{\tau}^{\top}\M\w_{\tau} &\geq & \w_{\tau}^{*\top}\M\w_{\tau}^* - 2\lambda_1(\M)\cdot\Vert{\w_{\tau} - \w_{\tau}^*}\Vert \\
&\geq & \w_{\tau}^{*\top}\M\w_{\tau}^* - \frac{\epsilon}{2}\lambda_1(\M), 
\end{eqnarray*}
where the last inequality follows from Lemma \ref{lem:approxProduct} and our choice of $\tilde{\epsilon}$.

Now, by our choice of $\tau$ and Theorem \ref{thm:pm}, we have that with probability at least $1-p$ it holds that
\begin{eqnarray*}
\w_{\tau}^{\top}\M\w_{\tau}  &\geq & (1-\epsilon/2)\lambda_1(\M) -\frac{\epsilon}{2}\lambda_1(\M) = (1-\epsilon)\lambda_1(\M).
\end{eqnarray*}

For the second item, note that
\begin{eqnarray*}
(\w_{\tau}^{\top}\u_1)^2 &=& (\w_{\tau}^{*\top}\u_1 + (\w_{\tau} - \w_{\tau}^{*\top})^{\top}\u_1)^2
\geq (\w_{\tau}^{*\top}\u_1)^2 - 2\Vert{\w_{\tau} - \w_{\tau}^*}\Vert \\
& \geq & (\w_{\tau}^{*\top}\u_1)^2 - \frac{\epsilon}{2} ,
\end{eqnarray*}
where the last inequality follows again from Lemma \ref{lem:approxProduct} and our choice of $\tilde{\epsilon}$.

Again, by our definition of $\tau$ and Theorem \ref{thm:pm}, we have that with probability at least $1-p$ it holds that
\begin{eqnarray*}
(\w_{\tau}^{\top}\u_1)^2 \geq 1- \epsilon/2 - \epsilon/2 = 1-\epsilon .
\end{eqnarray*}
\end{proof}

\subsection{Convergence of Algorithm \ref{alg:convexEV}}

The key step in the analysis of Algorithm \ref{alg:convexEV} is to show that if all numerical computations are carried out with sufficiently small error, then Algorithm \ref{alg:convexEV} successfully simulates the Power Method-based Algorithm \ref{alg:shrinkInvPM}. A main ingredient in Algorithm \ref{alg:shrinkInvPM} is the computations of the values $\Delta_s$ which are used in turn to update the estimates $\lambda_{(s)}$. Our use of the values $\Delta_s$ in Algorithm \ref{alg:shrinkInvPM} was through the approximation guarantees they provided for the gap $\lambda_{(s-1)}-\lambda_1$ (recall we have seen that $\frac{1}{2}(\lambda_{(s-1)}-\lambda_1) \leq \Delta_s \leq \lambda_{(s-1)} - \lambda_1$). The following lemma shows that with a correct choice for the numerical error parameter $\tilde{\epsilon}$, these guarantees also hold for the values $\Delta_s$ computed in Algorithm \ref{alg:convexEV}.

\begin{lemma}\label{lem:algEquiv}
Suppose that $\hat{\delta} \in\left[{\frac{\delta}{2}, \frac{3\delta}{4}}\right]$ and that $\tilde{\epsilon} \leq \frac{1}{16}\left({\frac{\hat{\delta}}{8}}\right)^{m_1+1}$, where $m_1$ is as defined in Algorithm \ref{alg:convexEV}. Then with probability at least $1-p$ it holds that for any $s$, the update of the variables $\Delta_s, \lambda_{(s)}$ in Algorithm \ref{alg:convexEV} is equivalent to that in Algorithm \ref{alg:shrinkInvPM}. In particular, for all $s\geq 1$ it holds that
\begin{eqnarray*}
\frac{1}{2}(\lambda_{(s-1)}-\lambda_1) \leq \Delta_s \leq \lambda_{(s-1)}-\lambda_1 .
\end{eqnarray*}
\end{lemma}

\begin{proof}
For clarity we refer by $\Delta_s, \lambda_{(s)}$ to the values computed in Algorithm \ref{alg:convexEV} and by $\tilde{\Delta}_s, \tilde{\lambda}_{(s)}$ to the corresponding values computed in Algorithm \ref{alg:shrinkInvPM} from Section \ref{sec:PowerMethodBasedAlg}.
The proof of the lemma is by induction on $s$. For the base case $s=0$, clearly $\lambda_{(0)}=\tilde{\lambda}_{(0)}$ and both $\Delta_0, \tilde{\Delta}_0$ are undefined.
Consider now some $s \geq 1$.

Suppose for now that 
\begin{eqnarray}\label{eq:numericEpsilon}
\tilde{\epsilon} \leq \min\{\frac{1}{16\Gamma(\M_s^{-1},m_1)}, \frac{\lambda_1(\M_s^{-1})}{8}\} . 
\end{eqnarray}
Then it follows from 
Theorem \ref{thm:approxPM} and our choice of $m_1$ that with probability at least $1-p$,
\begin{eqnarray}\label{eq:4}
\w_s^{\top}\M_s^{-1}\w_s \geq \frac{3}{4}\lambda_1(\M_s^{-1}).
\end{eqnarray}

By the definition of the vector $\v_s$ in Algorithm \ref{alg:convexEV} and the Cauchy-Schwartz inequality it holds that
\begin{eqnarray}\label{eq:5}
\w_s^{\top}\v_s = \w_s^{\top}\M_s^{-1}\w_s + \w_s^{\top}(\v_s - \M_s^{-1}\w_s) \in [\w_s^{\top}\M_s^{-1}\w_s -\tilde{\epsilon}, \w_s^{\top}\M_s^{-1}\w_s + \tilde{\epsilon}] .
\end{eqnarray}

Thus, combining Eq. \eqref{eq:4} and \eqref{eq:5}, we have that
\begin{eqnarray*}
\w_s^{\top}\v_s -\tilde{\epsilon}\in[\w_s^{\top}\M_s^{-1}\w_s -2\tilde{\epsilon}, \w_s^{\top}\M_s^{-1}\w_s] 
\subseteq [3\lambda_1(\M_s^{-1})/4 - 2\tilde{\epsilon}, \lambda_1(\M_s^{-1})] .
\end{eqnarray*}

By our choice of $\tilde{\epsilon}$ it follows that
\begin{eqnarray*}
\w_s^{\top}\v_s -\tilde{\epsilon} \in [\lambda_1(\M_s^{-1})/2, \lambda_1(\M_s^{-1})] .
\end{eqnarray*}

Thus, the computation of the value of $\Delta_s$, and as a result the computation of $\lambda_{(s)}$, is identical to that of $\tilde{\Delta}_s, \tilde{\lambda}_{(s)}$, and the claim follows.

It only remains to give an explicit bound on $\tilde{\epsilon}$, as defined in Eq. \eqref{eq:numericEpsilon}. For this we need to upper bound $\Gamma(\M_s^{-1},m_1)$ for all values of $s$.
Recall that from the results of Section \ref{sec:PowerMethodBasedAlg} it follows that the values $\lambda_{(s)}$ are monotonically non-increasing (see Eq. \eqref{eq:gapDec} in proof of Lemma \ref{lem:loopBound}) and lower-bounded by $\lambda_{(f)} \geq \lambda_1+\hat{\delta}/4$ (Lemma \ref{lem:finalGapBound}). By our assumption of $\hat{\delta}$ we have that
\begin{eqnarray}\label{eq:8}
\lambda_1(\M_s^{-1}) &=& \frac{1}{\lambda_{(s-1)} - \lambda_1} \geq \frac{1}{\lambda_{(0)}-\lambda_1} = \frac{1}{1+\hat{\delta}-\lambda_1} 
\geq \frac{1}{1+\frac{3\delta}{4}-\delta} \nonumber \\
&>& 1 + \frac{\delta}{4} \geq 1 + \frac{\hat{\delta}}{3}, 
\end{eqnarray}
where the second inequality holds since by definition of $\delta$ it follows that $\lambda_1 = \lambda_2 + \delta \geq \delta$.

Fix a natural number $t$. By the definition of $\Gamma(\M,t)$ (see Lemma \ref{lem:approxProduct}) we have that 
\begin{eqnarray}\label{eq:9}
\Gamma(\M_s^{-1},t) &=& \frac{2}{\lambda_d(\M_s^{-1})^t}\cdot\frac{\lambda_1(\M_s^{-1})^t-1}{\lambda_1(\M_s^{-1})-1}
= \frac{2}{\lambda_1(\M_s^{-1})-1}\left({\frac{\lambda_1(\M_s^{-1})}{\lambda_d(\M_s^{-1})}}\right)^t \nonumber \\
&\leq & \frac{6}{\hat{\delta}}\left({\frac{\lambda_{(s-1)}-\lambda_d}{\lambda_{(s-1)}-\lambda_1}}\right)^t
\leq \frac{6}{\hat{\delta}}\left({\frac{\lambda_{(0)}}{\lambda_{(f)}-\lambda_1}}\right)^t \nonumber \\
&\leq& \frac{6}{\hat{\delta}}\left({\frac{1+\hat{\delta}}{\hat{\delta}/4}}\right)^t = \frac{6}{\hat{\delta}}\left({4+\frac{4}{\hat{\delta}}}\right)^t < \left({\frac{8}{\hat{\delta}}}\right)^{t+1},
\end{eqnarray}
where the first inequality follows from Eq. \eqref{eq:8}, the second inequality follows from the bounds $\lambda_d \geq 0$ and $\lambda_{(f)} \leq \lambda_{(s-1)} \leq \lambda_{(0)}$, and the third inequality follows from plugging the value of $\lambda_{(0)}$ and from Lemma \ref{lem:finalGapBound}.

Thus, it follows that we can take $\tilde{\epsilon} \leq \min\{\frac{1}{16}\left({\frac{\hat{\delta}}{8}}\right)^{m_1+1}, \frac{1+\hat{\delta}/3}{8}\} = \frac{1}{16}\left({\frac{\hat{\delta}}{8}}\right)^{m_1+1}$. 

\end{proof}

\begin{theorem}[Convergence of Algorithm \ref{alg:convexEV}]\label{thm:convConvexPCA}
Suppose that $\hat{\delta}\in[\delta/2,3\delta/4]$. Fix $\epsilon > 0$ . Let $m_1 \geq \TpmCrude(1/8,p)$ and $m_2 \geq \TpmAcc(3, \epsilon/2,p)$.
Suppose that $\tilde{\epsilon}$, satisfies that
\begin{eqnarray*}
\tilde{\epsilon} \leq \min\{\frac{1}{16}\left({\frac{\hat{\delta}}{8}}\right)^{m_1+1}, \frac{\epsilon}{4}\left({\frac{\hat{\delta}}{8}}\right)^{m_2+1}\}.
\end{eqnarray*}
Then, with probability at least $1-p$ it holds that the output of Algorithm \ref{alg:convexEV}, the unit vector $\w_f$, satisfies that
\begin{eqnarray*}
(\w_f^{\top}\u_1)^2 \geq 1-\epsilon,
\end{eqnarray*}
and the total number of calls to the convex optimization oracle $\mA$ is 
\begin{eqnarray*}
O\left({\log(d/p)\log(\delta^{-1}) + \log(\frac{d}{p\epsilon})}\right) .
\end{eqnarray*}
\end{theorem}

\begin{proof}
First, note that as in the proof of Theorem \ref{thm:basicAlgConv}, since all the noisy simulations of the Power Method in Algorithm \ref{alg:convexEV} are initialized with the same random unit vector $\hat{\w}_0$, they all succeed together with probability at least $1-p$ (provided that the other parameters $m_1,m_2, \tilde{\epsilon}$ are set correctly).

By our choice of $\tilde{\epsilon}$ and Lemma \ref{lem:algEquiv} it follows that we can invoke the results of Section \ref{sec:PowerMethodBasedAlg}.
By Lemma \ref{lem:loopBound}, we can upper bound the number of iterations made by the \textit{repeat-until} loop by $O(\log(\delta^{-1})$. Since each iteration of the loop requires $m_1+1$ calls to the optimization oracle $\mA$, the overall number of calls to $\mA$ during the loop is $O(m_1\log(\delta^{-1}))$.

By Lemma \ref{lem:finalGapBound} we have that the final estimate $\lambda_{(f)}$ satisfies that $\lambda_{(f)}-\lambda_1 \leq \frac{3\hat{\delta}}{2} \leq \frac{9\delta}{8} < 2\delta$. Thus, by Lemma \ref{lem:inverseCond} we have that
\begin{eqnarray}\label{eq:7}
\kappa(\M_f^{-1}) = \frac{\lambda_1(\M_f^{-1})}{\delta(\M_f^{-1})} \leq 3.
\end{eqnarray}
Suppose now that $\tilde{\epsilon} \leq \frac{\epsilon}{4\Gamma(\M_f^{-1}, m_2)}$.
By our choice of $m_2$ and Eq. \eqref{eq:7}, it follows from Theorem \ref{thm:approxPM} that with probability at least $1-p$ indeed $(\w_f^{\top}\u)^2 \geq 1-\epsilon$. 

The number of calls to the oracle $\mA$ in this final part is $m_2$. Thus overall number of calls to $\mA$ in Algorithm \ref{alg:convexEV} is $O(m_1\log(\delta^{-1}) + m_2)$.

It remains to lower-bound the term $\frac{\epsilon}{4\Gamma(\M_f^{-1}, m_2)}$.
Following the analysis in the proof of Lemma \ref{lem:algEquiv} (Eq. \eqref{eq:9}), we can upper bound $\Gamma(\M_f^{-1}, m_2) \leq \left({\frac{8}{\hat{\delta}}}\right)^{m_2+1}$, which agrees with the bound on $\tilde{\epsilon}$ stated in the theorem.

\end{proof}

In order to analyze the arithmetic complexity of Algorithm \ref{alg:convexEV} using a specific implementation for the optimization oracle $\mA$, it is not only important to bound the number of calls to $\mA$ (as done in Theorem \ref{thm:convConvexPCA}), but to also bound important parameters of the optimization problem \eqref{eq:pcaOptProblem} that naturally arise when considering the arithmetic complexity of different implementations for $\mA$. For this issue we have the following lemma which follows directly from the discussion in Section \ref{sec:PowerMethodBasedAlg} and the assumptions stated in Theorem \ref{thm:convConvexPCA}.

\begin{lemma}\label{lem:convexProblemsBounds}
Let $\lambda,\w$ be such that during the run of Algorithm \ref{alg:convexEV}, the optimization oracle $\mA$ is applied to the minimization of the function
\begin{eqnarray*}
F_{\w,\lambda}(\z) = \frac{1}{2}\z^{\top}(\lambda\I-\X)\z - \w^{\top}\z.
\end{eqnarray*}
Then, under the conditions stated in Theorem \ref{thm:convConvexPCA} it holds that
\begin{enumerate}
\item $F_{\w,\lambda}(\z)$ is $(\lambda-\lambda_1(\X))=\Omega(\delta)$-strongly convex.
\item for all $i\in[n]$ it holds that the function $f_i(\z) = \frac{1}{2}\z^{\top}(\lambda\I-\x_i\x_i^{\top})\z - \w^{\top}\z$ is $1+\hat{\delta}=O(1)$-smooth.
\item $\log(\Vert{\z^*}\Vert) = \tilde{O}(1)$, where $\z^*$ is the global minimizer of $F_{\w,\lambda}(\z)$.
\end{enumerate}
\end{lemma}



\section{Putting it all together: Fast PCA via SVRG}\label{sec:SVRG4PCA}

In this section we prove Theorems \ref{thm:main}, \ref{thm:mainAcc}.

Following the convex optimization-based eigenvector algorithm presented in the previous section - Algorithm \ref{alg:convexEV}, we consider the implementation of the convex optimization oracle $\mA$, invoked in Algorithm \ref{alg:convexEV}, using the \textsc{SVRG} algorithm \cite{Zhang13} discussed in subsection \ref{sec:stochasticOptPrem}. Recall that the oracle $\mA$ is used to solve Problem \eqref{eq:pcaOptProblem} for some parameters $\lambda, \w$. Indeed the objective $F_{\w, \lambda}(\z)$ in \eqref{eq:pcaOptProblem} could be written as a finite some of functions in the following way:

\begin{eqnarray}\label{eq:pca-optprob}
F_{\w,\lambda}(\z) = 
\frac{1}{n}\sum_{i=1}^n\left({\frac{1}{2}\z^{\top}(\lambda\I-\x_i\x_i^{\top})\z - \w^{\top}\z}\right) .
\end{eqnarray}
Further, recall that for $\lambda > \lambda_1(\X)$, $F_{\w,\lambda}(\z)$ is always $(\lambda-\lambda_1(\X))$-strongly-convex and that for every $i\in[n]$, the function
\begin{eqnarray*}
f_i(\z) := \frac{1}{2}\z^{\top}(\lambda\I-\x_i\x_i^{\top})\z - \w^{\top}\z ,
\end{eqnarray*}
is $\max\{\lambda,\Vert{\x_i}\Vert^2-\lambda\}$ smooth. However, $f_i(\z)$ need not be convex. Hence the SVRG theorem from \cite{Zhang13} could not be directly applied to minimizing \eqref{eq:pca-optprob}. However, we prove in the appendix that the SVRG method still converges but with a slightly worse dependence on the condition number \footnote{We note that in \cite{SSS15} it was shown that the same kind of result holds also for the \textsc{SDCA} method that was originally introduced in \cite{ShwartzZhang}.}.

Below we give an explicit implementation of the the SVRG algorithm for minimizing \eqref{eq:pca-optprob}.

\begin{algorithm}
\caption{\textsc{SVRG for PCA}}
\label{alg:svrg4pca}
\begin{algorithmic}[1]
\STATE Input: $\lambda\in\mathbb{R}, \X=\frac{1}{n}\sum_{i=1}^n\x_i\x_i^{\top}$, $\w$, $\eta, m, T$.
\STATE $\tilde{\z}_0 \gets \vec{0}$
\FOR{$s=1,...T$}
\STATE $\tilde{\z} \gets \tilde{\z}_{s-1}$
\STATE $\tilde{\mu} \gets (\lambda\I-\X)\tilde{\z} - \w_{t-1}$
\STATE $\z_0 \gets \tilde{\z}$
\FOR{$t=1,2,...,m$}
\STATE Randomly pick $i_t\in[n]$
\STATE $\z_t \gets \z_{t-1} - \eta\left({(\lambda\I-\x_{i_t}\x_{i_t}^{\top})(\z_{t-1} - \tilde{\z}) + \tilde{\mu}}\right)$ 
\ENDFOR
\STATE $\tilde{\z}_s \gets \frac{1}{m}\sum_{t=0}^{m-1}\z_t$
\ENDFOR
\RETURN $\tilde{\z}_T$
\end{algorithmic}
\end{algorithm}

The following theorems are proven in the appendix.
\begin{theorem}\label{thm:svrg4pca}
Fix $\epsilon > 0, p > 0$. There exists a choice of $\eta,m$ such that Algorithm \ref{alg:svrg4pca} finds with probability at least $1-p$ an $\epsilon$-approxiated minimizer of \eqref{eq:pca-optprob} in overall time
\begin{eqnarray*}
\tilde{O}\left({N +\frac{d}{(\lambda-\lambda_1(\X))^2}}\right) .
\end{eqnarray*}
\end{theorem}

Based on the recent acceleration framework of \cite{Harchaoui15} we also have the following result (the proof is given in the appendix).
\begin{theorem}\label{thm:Accsvrg4pca}
Fix $\epsilon > 0, p > 0$. Assume that $\lambda-\lambda_1 = O(\sqrt{d/N})$. There exists an \textit{accelerated} version of Algorithm \ref{alg:svrg4pca} that finds with probability at least $1-p$ an $\epsilon$-approximated minimizer of \eqref{eq:pca-optprob} in overall time
\begin{eqnarray*}
\tilde{O}\left({\frac{N^{3/4}d^{1/4}}{\sqrt{\lambda-\lambda_1(\X)}}}\right) .
\end{eqnarray*}
\end{theorem}

\subsection{Proving Theorems \ref{thm:main}, \ref{thm:mainAcc}}

The proof of Theorem \ref{thm:main} follows from the bounds in Theorem \ref{thm:convConvexPCA}, Lemma \ref{lem:convexProblemsBounds} and Theorem \ref{thm:svrg4pca}.

The proof of Theorem \ref{thm:mainAcc} follows from the bounds in Theorem \ref{thm:convConvexPCA}, Lemma \ref{lem:convexProblemsBounds} and Theorem \ref{thm:Accsvrg4pca}.

\section{Proof of Theorems \ref{thm:main:gapfree}, \ref{thm:main:gapfreeAcc}}\label{sec:gapfree}

In this section we prove Theorems \ref{thm:main:gapfree} and \ref{thm:main:gapfreeAcc}. Since the algorithm and corresponding proofs basically follow from the results of the previous sections, we give the algorithm and a sketch of the proofs.

The algorithm is the same as Algorithm \ref{alg:convexEV}, but intuitively, it replaces the estimate for the eigengap $\hat{\delta}$ with the desired approximation error for the top eigenvalue, $\epsilon$. This is intuitive since now, instead of finding a unit vector that is approximately entirely contained in the span of vectors $\{\u_i \, | \, \lambda_i > \lambda_1 - \delta\}=\{\u_1\}$, we wish to find a unit vector that is approximately entirely contained in the span of vectors $\{\u_i \, | \, \lambda_i > \lambda_1 - \epsilon\}$.

\begin{algorithm}
\caption{\textsc{Leading Eigenvalue Approximation via Convex Optimization}}
\label{alg:convexEigVal}
\begin{algorithmic}[1]
\STATE Input: matrix $\X\in\mathbb{R}^{n\times n}$ such that $\X\succeq 0$, $\lambda_1(\X) \leq 1$, accuracy parameter $\epsilon\in(0,1)$, failure probability parameters $p$, positive integers $m_1,m_2$, numerical accuracy parameter $\tilde{\epsilon}$
\STATE $\lambda_{(0)} \gets 1+ \epsilon$
\STATE Let $\hat{\w}_0$ be a random unit vector
\STATE $s \gets 0$
\REPEAT
\STATE $s \gets s+1$
\STATE Let $\M_s = (\lambda_{(s-1)}\I-\X)$
\FOR{$t=1...m_1$}
\STATE Apply Algorithm $\mA$ to find a vector $\hat{\w}_t$ such that $\Vert{\hat{\w}_t - \M_s^{-1}\hat{\w}_{t-1}}\Vert \leq \tilde{\epsilon}$
\ENDFOR
\STATE $\w_s \gets \frac{\hat{\w}_{m_1}}{\Vert{\hat{\w}_{m_1}}\Vert}$
\STATE Apply Algorithm $\mA$ to find a vector $\v_s$ such that $\Vert{\v_s - \M_s^{-1}\w_s}\Vert \leq \tilde{\epsilon}$
\STATE $\Delta_s \gets \frac{1}{2}\cdot\frac{1}{ \w_s^{\top}\v_s - \tilde{\epsilon}}$
\STATE $\lambda_{(s)} \gets \lambda_{(s-1)} - \frac{\Delta_s}{2}$
\UNTIL{$\Delta_s \leq \epsilon$}
\STATE $\lambda_{(f)} \gets \lambda_{(s)}$
\STATE Let $\M_f = (\lambda_{(f)}\I-\X)$
\FOR{$t=1...m_2$}
\STATE Apply Algorithm $\mA$ to find a vector $\hat{\w}_t$ such that $\Vert{\hat{\w}_t - \M_f^{-1}\hat{\w}_{t-1}}\Vert \leq \tilde{\epsilon}$
\ENDFOR
\RETURN $\w_f \gets \frac{\hat{\w}_{m_2}}{\Vert{\hat{\w}_{m_2}}\Vert}$
\end{algorithmic}
\end{algorithm}

The following simple lemma is of the same flavor as Lemma \ref{lem:inverseCond} and shows that we can benefit from conditioning the inverse matrix $(\lambda\I-\X)^{-1}$, even if the goal is only to approximate the leading eigenvalue and not the leading eigenvector.

\begin{lemma}\label{lem:inverseCond2}
Fix $\epsilon > 0$ and a scalar $a > 0$.  Let $\M^{-1} = (\lambda\I-\X)^{-1}$ such that $\lambda_1(\X) + a\cdot \epsilon \geq \lambda > \lambda_1(\X)$. It holds for all $i\in[d]$ such that $\lambda_i(\X) \leq \lambda_1(\X)-\epsilon$, that
\begin{eqnarray*}
\frac{\lambda_1(\M^{-1})}{\lambda_i(\M^{-1})} \geq 1 + a^{-1} .
\end{eqnarray*}
\end{lemma}

\begin{proof}

It holds that
\begin{eqnarray*}
\frac{\lambda_1(\M^{-1})}{\lambda_i(\M^{-1})} = \frac{1}{\lambda-\lambda_1} \cdot (\lambda-\lambda_i)
= 1 + \frac{\lambda_1-\lambda_i}{\lambda-\lambda_i} \geq 1 + \frac{\lambda_1-\lambda_i}{a\epsilon} .
\end{eqnarray*}

Thus, for any $i\in[d]$ such that $\lambda_1-\lambda_i \geq \epsilon$ it follows that $\frac{\lambda_1(\M^{-1})}{\lambda_i(\M^{-1})} \geq 1 + a^{-1}$.
\end{proof}

In order to prove the convergence of Algorithm \ref{alg:convexEigVal} we are going to need a slightly different result for the Power Method (Algorithm \ref{alg:powerm}) than that of Theorem \ref{thm:pm}. This result follows from the same analysis as Theorem \ref{thm:pm}. For a proof see the proof of Theorem \ref{thm:pm}.

\newcommand{\Tpm}{T^{\textsc{PM}}}

\begin{lemma}[Convergence of Power Method to span of top eigenvectors]\label{lem:pm:span}
Let $\M$ be a positive definite matrix and denote its eigenvalues in descending order by $\lambda1, \lambda_2,...,\lambda_d$, and let $\u_1,\u_2,...,\u_d$ denote the corresponding eigenvectors. Fix $\epsilon_1,\epsilon_2 \in (0,1)$ and failure probability $p > 0$. Define:
\begin{eqnarray*}
\Tpm(\epsilon_1,\epsilon_2,p) := \lceil{\frac{1}{2\epsilon_1}\ln\left({\frac{9d}{p^2\epsilon_2}}\right)}\rceil.
\end{eqnarray*}
Then, with probability $1-p$ it holds for any $t \geq \Tpm(\epsilon_1,\epsilon_2,p)$ that 
\begin{eqnarray*}
\sum_{i\in[d]: \, \lambda_i \leq (1-\epsilon_1)\lambda_1}(\w_t^{\top}\u_i)^2 \leq \epsilon_2 .
\end{eqnarray*}
The probability of success depends only on the random variable $(\w_0^{\top}\u_1)^2$.
\end{lemma}

\begin{theorem}\label{thm:convergenceEigValAlg}
There exists a choice for the parameters $m_1,m_2,\tilde{\epsilon}$ such that Algorithm \ref{alg:convexEigVal}, when implemented with SVRG as the optimization oracle $\mA$, finds in time $\tilde{O}\left({\frac{d}{\epsilon^2} +N}\right)$ a unit vector $\w_f$ that with probability at least $1-p$ satisfies,
\begin{eqnarray*}
\w_f^{\top}\X\w_f \geq \lambda_1 - \epsilon .
\end{eqnarray*}
\end{theorem}

\begin{proof}[Proof sketch]
Following the same lines of the analysis presented in Theorem \ref{thm:convConvexPCA}, there exists $m_1 = \tilde{O}(1)$ and $\tilde{\epsilon}$ satisfying $\log(1/\tilde{\epsilon}) = \tilde{O}(1)$, such that the repeat-until loop terminates after $\tilde{O}(1)$ iterations with a value $\lambda_{(f)}$ such that $\frac{3}{2}\epsilon \geq \lambda_{(f)} - \lambda_1 \geq \frac{1}{4}\epsilon$.

Denote $S(\epsilon) = \{i\in[d] \, | \, \lambda_i > \lambda_1 - \epsilon/2\}$. By Lemma \ref{lem:inverseCond2} we have that for all $i\notin S(\epsilon)$ it holds that $\lambda_i(\M_f^{-1}) \leq \frac{3}{4}\lambda_1(\M_f^{-1})$.

Thus, by applying Lemma \ref{lem:pm:span} with respect to the matrix $\M_f^{-1}$ with $\epsilon_1 = \frac{1}{4}$, $\epsilon_2 = \epsilon/2$, we get that for $m_2 = \tilde{O}(1)$, it follows that $\sum_{i\in S(\epsilon)}(\w_f^{\top}\u_i)^2 \geq 1-\epsilon/2$. Thus it follows that
\begin{eqnarray*}
\w_f^{\top}\X\w_f &=& \sum_{i=1}^d\lambda_i(\w_f^{\top}\u_i)^2 \geq \sum_{i\in S(\epsilon)}\lambda_i(\w_f^{\top}\u_i)^2
\geq (\lambda_1-\epsilon/2)\sum_{i\in S(\epsilon)}(\w_f^{\top}\u_i)^2 \\
&\geq & (\lambda_1-\epsilon/2)(1-\epsilon/2) > \lambda_1-\epsilon .
\end{eqnarray*}

Since on every iteration $s$ of the repeat-until loop it holds that $\lambda_{(s)}-\lambda_1 = \Omega(\epsilon)$, similarly to Lemma \ref{lem:convexProblemsBounds}, we have that the optimization oracle $\mA$ is applied to solving $O(1)$-smooth and $\Omega(\epsilon)$-strongly convex problems. Hence, by the SVRG theorem, Theorem \ref{thm:svrg4pca}, each invocation of $\mA$ when implemented using the SVRG algorithm, requires $\tilde{O}\left({\frac{d}{\epsilon^2} + N}\right)$ time. Hence, the theorem follows.
\end{proof}

Theorem \ref{thm:main:gapfree} follows from Theorem \ref{thm:convergenceEigValAlg} and the observation that by using a standard concentration argument for matrices \footnote{See for instance the Matrix Hoeffding concentration inequality in \cite{Tropp12}.}, instead of applying the algorithm directly to the matrix $\X = \frac{1}{n}\sum_{i=1}^n\x_i\x_i$, it suffices to apply it to the matrix $\tilde{\X} = \frac{1}{n'}\sum_{i=1}^{n'}\tilde{\x}_i\tilde{\x}_i$, where $n' =\tilde{O}(\epsilon^{-2})$ and the vectors $\{\tilde{\x}_i\}_{i=1}^{n'}$ are sampled uniformly from $\{\x_i\}_{i=1}^n$. Hence, we can basically substitute $N$ with $d/\epsilon^2$ in the running time stated in Theorem \ref{thm:convergenceEigValAlg}. 
From this observation it also follows that Theorem \ref{thm:main:gapfree} holds also when $\X$ is not given explicitly as a finite sum of rank one matrices, but $\X = \E_{\x\sim\mD}[\x\x^{\top}]$ for some unknown distribution $\mD$ from which we can sample in $O(d)$ time.

Theorem \ref{thm:main:gapfreeAcc} follows if we replace the use of the SVRG algorithm in the proof of Theorem \ref{thm:convergenceEigValAlg} with its accelerated version, i.e. use Theorem \ref{thm:Accsvrg4pca} instead of Theorem \ref{thm:svrg4pca}.

\section{Dense PCA and Acceleration of Algorithms for Semidefinite Optimization}\label{sec:sdpApp}

So far we have considered the problem of computing an approximation for the largest eigenvector of a matrix $\X$ given as a sum of rank-one matrices, i.e. $\X = \frac{1}{n}\sum_{i=1}^n\x_i\x_i^{\top}$. However, our techniques extend well beyond this case. In fact, the only place in which we have used our structural assumption on $\X$ is in our implementation of SVRG for PCA given in Algorithm \ref{alg:svrg4pca} and the corresponding runtime analysis.

In this section we consider the following natural generalization of the PCA problem which we refer to as \textit{dense PCA}: we assume $\X$ is of the form 
\begin{eqnarray*}
\X = \sum_{i=1}^np_i\A_i, 
\end{eqnarray*}
where $\sum_{i=1}^np_i = 1$, $\forall i\in[n] p_i \geq 0$ and $\forall i\in[n]$ $\A_i$ is a symmetric real $d\times d$ matrix. We further assume that $\forall i\in[n]$ $\Vert{\A_i}\Vert \leq 1$ and $\X \succeq 0$. We focus here on the problem of fast approximation of $\lambda_1(\X)$, i.e. finding a unit vector $\w$ such that
\begin{eqnarray*}
\w^{\top}\X\w \geq \lambda_1(\X) - \epsilon .
\end{eqnarray*}

Attaining such an approximated eignevector for $\X$, as considered in this section ,is an important subroutine in several algorithms for Semidefinite Programming \cite{AHK05, Hazan08, Daspremont14, GHsubSDP15} and Online Learning \cite{GHM15, KW15, ALT15, DTTZ14}.

For the purposes of this section we denote by $N$ the total number of non-zeros in all matrices $\A_1,...,\A_n$, and in addition we denote by $S$ the maximal number of non-zero entries in any of the matrices $\A_1,...,\A_n$.

\subsection{SVRG for Dense PCA}
As discussed above, our fast PCA algorithms and corresponding analysis could be directly applied to the dense case under consideration in this section. The only thing that needs to be changed is the application of the SVRG algorithm. A modified SVRG algorithm for the dense PCA problem is detailed below. Specifically, we apply SVRG to the following optimization problem:
\begin{eqnarray}\label{eq:densePCA-optprob}
\min_{\z\in\reals^d}\{F_{\w,\lambda}(\z) := \sum_{i=1}^np_i\left({\z^{\top}(\lambda\I-\A_i)\z - \w^{\top}\z}\right) \}.
\end{eqnarray}

\begin{algorithm}[H]
\caption{SVRG for Dense PCA}
\label{alg:svrg4DensePCA}
\begin{algorithmic}[1]
\STATE Input: $\lambda\in\mathbb{R}, \X=\sum_{i=1}^np_i\A_i$, $\w$, $\eta, m, T$.
\STATE $\tilde{\z}_0 \gets \vec{0}$
\FOR{$s=1,...T$}
\STATE $\tilde{\z} \gets \tilde{\z}_{s-1}$
\STATE $\tilde{\mu} \gets (\lambda\I-\X)\tilde{\z} - \w_{t-1}$
\STATE $\z_0 \gets \tilde{\z}$
\FOR{$t=1,2,...,m$}
\STATE Pick $i_t\in[n]$ according to probability distribution $(p_1,p_2,...,p_n)$
\STATE $\z_t \gets \z_{t-1} - \eta\left({(\lambda\I-\A_{i_t})(\z_{t-1} - \tilde{\z}) + \tilde{\mu}}\right)$ 
\ENDFOR
\STATE $\tilde{\z}_s \gets \frac{1}{m}\sum_{t=0}^{m-1}\z_t$
\ENDFOR
\RETURN $\tilde{\z}_T$
\end{algorithmic}
\end{algorithm}

Note that here, as opposed to the standard PCA problem, we assume that $\X$ is given by a weighted average of matrices and not necessarily a uniform average. This difference comes into play in Algorithm \ref{alg:svrg4DensePCA} when we sample a random gradient index $i_t$ according to the weights $\{p_i\}_{i=1}^n$, and not uniformly as in Algorithm \ref{alg:svrg4pca}. This change however does not change the analysis of the algorithm given in Theorem \ref{thm:svrg_nonconvex}, and thus we have the following theorem which is analogues to Theorem \ref{thm:svrg4pca}.

\begin{theorem}\label{thm:svrg4DensePCA}
Fix $\epsilon > 0, p > 0$. There exists a choice of $\eta,m$ such that Algorithm \ref{alg:svrg4DensePCA} finds with probability at least $1-p$ an $\epsilon$-approxiated minimizer of \eqref{eq:densePCA-optprob} in overall time
\begin{eqnarray*}
\tilde{O}\left({N +\frac{d+S}{(\lambda-\lambda_1(\X))^2}}\right) .
\end{eqnarray*}
\end{theorem}

By Plugging Theorem \ref{thm:svrg4DensePCA} into the proof of Theorem \ref{thm:main:gapfree}(instead of Theorem \ref{thm:svrg4pca}) we arrive at the following theorem (analogues to Theorem \ref{thm:main:gapfree}).

\begin{theorem}\label{thm:DensePCA}
Fix $\epsilon >0, p>0$. There exists an algorithm that finds with probability at least $1-p$ a unit vector $\w$ such that $\w^{\top}\X\w \geq \lambda_1 - \epsilon$, in total time $\tilde{O}\left({\frac{d+S}{\epsilon^2}}\right)$.
\end{theorem}

\subsection{Faster sublinear-time SDP algorithm}

Here we detail a specific application of Theorem \ref{thm:DensePCA} to accelerate the sublinear-time approximation algorithm for Semidefinite Programming recently proposed by Garber and Hazan \cite{GHsubSDP15}.
Towards this end we consider the following semidefinite optimization problem:
\begin{eqnarray}\label{eq:sdpopt}
\max_{\W: \, \W \succeq 0, \, \textrm{Trace}(\W)=1}\,\min_{i\in[n]}\W \bullet \A_i - b_i,
\end{eqnarray}

where we assume that $\forall i\in[n]$: $\A_i$ is symmetric, $\Vert{\A_i}\Vert \leq 1$, $\Vert{\A_i}\Vert_F \leq F$ for some $F$, and $\vert{b_i}\vert \leq 1$. Note that under the assumption $\Vert{\A_i}\Vert \leq 1$ it holds that $F \leq \sqrt{d}$. We use $\A\bullet \B$ to denote the dot product $\sum_{i,j}\A_{i,j}\B_{i,j}$.

As before we denote by $N$ the total number of non-zero entries the matrices $\A_1,...,\A_n$ and by $S$ the maximal number of non-zero in any single matrix.

To the best of our knowledge the algorithm in \cite{GHsubSDP15} is the current state-of-the-art for approximating Problem  \eqref{eq:sdpopt} for a wide regime of the input parameters $d,n,\epsilon$.

\begin{theorem}[Theorem 1 in \cite{GHsubSDP15}]\label{thm:GHSDP}
There exists an algorithm that finds in total time $$\tilde{O}\left({\frac{1}{\epsilon^2}\left({mF^2 + \min\{\frac{S}{\epsilon^{2.5}},\, \frac{N}{\sqrt{\epsilon}}\}}\right)}\right)$$
a solution $\W_f$ such that with probability at least $1/2$ it holds that
$$\min_{i\in[n]}\W_f \bullet \A_i - b_i \geq \max_{\W: \, \W \succeq 0, \, \textrm{Trace}(\W)=1}\,\min_{i\in[n]}\W \bullet \A_i - b_i - \epsilon .$$ 
\end{theorem}

The algorithm performs roughly $\epsilon^{-2}$ iterations where each iteration is comprised of two main parts: i) low-variance estimation of the products $\A_i \bullet \w\w^{\top}$ $\forall i\in[n]$ for some unit vector $\w\in\reals^d$, which are used to obtain a probability distribution $p\in\reals^n$ over the functions $\{f_i(\W):=\A_i \bullet \W - b_i \}_{i=1}^n$, and ii) an $O(\epsilon)$ - approximated leading eigenvalue computation of the matrix $\sum_{i=1}^np_i\A_i$ where $p$ is a distribution as discussed above. The term on the right in the running time stated in Theorem \ref{thm:GHSDP} comes from this approximated eigenvalue computation which is done according to either the standard \textit{Lanczos method} or the \textit{Sample Lanczos method} detailed in Table \ref{table:prevwork:gapfree}. By replacing the Sample Lanczos procedure with Theorem \ref{thm:DensePCA} we arrive at the following improved Theorem.

\begin{theorem}[Accelerated sublinear SDP solver]
There exists an algorithm that finds in total time $$\tilde{O}\left({\frac{1}{\epsilon^2}\left({mF^2 + \min\{\frac{S}{\epsilon^{2}},\, \frac{N}{\sqrt{\epsilon}}\}}\right)}\right)$$
a solution $\W_f$ such that with probability at least $1/2$ it holds that
$$\min_{i\in[n]}\W_f \bullet \A_i - b_i \geq \max_{\W: \, \W \succeq 0, \, \textrm{Trace}(\W)=1}\,\min_{i\in[n]}\W \bullet \A_i - b_i - \epsilon .$$ 
\end{theorem}

A slight technical issue is that the queried matrix $\X=\sum_{i=1}^np_i\A_i$ need not be positive semidefinite as we assumed in our results, however under our assumptions we can easily apply our results to the matrix $\tilde{\X}=\sum_{i=1}^np_i\A_i + \I$ which is positive semidefinite, which only slightly affects the leading constants in our theorems.

\bibliographystyle{plain}
\bibliography{bib}

\appendix

\section{Convergence of the Power Method}

We first restate the theorem and then prove it.
\begin{theorem}
Let $\M$ be a positive definite matrix and denote its eigenvalues in descending order by $\lambda_1, \lambda_2,...,\lambda_d$, and let $\u_1,\u_2,...,\u_d$ denote the corresponding eigenvectors. Denote $\delta = \lambda_1 - \lambda_2$ and $\kappa = \frac{\lambda_1}{\delta}$. Fix an error tolerance $\epsilon > 0$ and failure probability $p > 0$. Define:
\begin{eqnarray*}
\TpmCrude(\epsilon,p) = \lceil{\frac{1}{\epsilon}\ln\left({\frac{18d}{p^2\epsilon}}\right)}\rceil,  \qquad
\TpmAcc(\kappa,\epsilon,p) = \lceil{\frac{\kappa}{2}\ln\left({\frac{9d}{p^2\epsilon}}\right)}\rceil .
\end{eqnarray*}
Then, with probability $1-p$ it holds that
\begin{enumerate}
\item (crude regime) $\forall t \geq \TpmCrude(\epsilon,p)$: $\w_t^{\top}\M\w_t \geq (1-\epsilon)\lambda_1$.
\item (accurate regime)  $\forall t \geq \TpmAcc(\kappa,\epsilon,p)$: $(\w_t^{\top}\u_1)^2 \geq 1-\epsilon$.
\end{enumerate}
In both cases, the success probability depends only on the random variable $(\w_0^{\top}\u_1)^2$.
\end{theorem}

\begin{proof}

By the update rule of Algorithm \ref{alg:powerm}, it holds for all $i\in[d]$ that

\begin{eqnarray}\label{eq:pm1}
(\w_t^{\top}\u_i)^2 &=& \frac{\left({(\M^t\w_0)^{\top}\u_i}\right)^2}{\Vert{\M^t\w_0}\Vert^2}
= \frac{(\w_0^{\top}\M^t\u_i)^2}{\w_0^{\top}\M^{2t}\w_0} = \frac{(\lambda_i^t\w_0^{\top}\u_i)^2}{\sum_{j=1}^d\lambda_j^{2t}(\w_0^{\top}\u_j)^2} 
= \frac{(\w_0^{\top}\u_i)^2}{\sum_{j=1}^d\left({\frac{\lambda_j}{\lambda_i}}\right)^{2t}(\w_0^{\top}\u_j)^2} \nonumber \\
& \leq & \frac{(\w_0^{\top}\u_i)^2}{\left({\frac{\lambda_1}{\lambda_i}}\right)^{2t}(\w_0^{\top}\u_1)^2}
 =  \frac{(\w_0^{\top}\u_i)^2}{(\w_0^{\top}\u_1)^2}\left({\frac{\lambda_i}{\lambda_1}}\right)^{2t} 
.
\end{eqnarray}

Since $\w_0$ is a random unit vector, according to Lemma 5 in \cite{AroraRV09}, it holds that with probability at least $1-p$, $(\w_0^{\top}\u_1)^2 \geq \frac{p^2}{9d}$. Thus we have that with probability at least $1-p$ it holds for all $i\in[d]$ that 

\begin{eqnarray}\label{eq:pm2}
(\w_t^{\top}\u_i)^2 \leq (\w_0^{\top}\u_i)^2\frac{9d}{p^2}\left({1 - \frac{\lambda_1-\lambda_i}{\lambda_1}}\right)^{2t}
\leq (\w_0^{\top}\u_i)^2\frac{9d}{p^2}\cdot\exp\left({-2\frac{\lambda_1-\lambda_i}{\lambda_1}t}\right) .
\end{eqnarray}

Given $\epsilon\in(0,1)$, define $S(\epsilon) = \lbrace{i\in[d] \, | \, \lambda_i > (1-\epsilon)\lambda_1}\rbrace$. Fix now $\epsilon_1,\epsilon_2\in(0,1)$ and define 
\begin{eqnarray*}
T(\epsilon_1,\epsilon_2,p) := \lceil{\frac{1}{2\epsilon_1}\ln\left({\frac{9d}{p^2}\cdot\frac{1}{\epsilon_2}}\right)}\rceil.
\end{eqnarray*}

According to Eq. \eqref{eq:pm2}, with probability at least $1-p$ we have that for $t \geq T(\epsilon_1,\epsilon_2,p)$, for all $i\notin S(\epsilon_1)$ it holds that $(\w_t^{\top}\u_i)^2 \leq \epsilon_2(\w_0^{\top}\u_i)^2$, and thus in particular it holds that $\sum_{i\in S(\epsilon_1)}(\w_t^{\top}\u_i)^2 \geq 1- \epsilon_2$.

Part one of the theorem now follows by noticing that according to the above, by setting $\epsilon_1 = \epsilon_2 = \epsilon/2$ we have that with probability at least $1-p$, for $t\geq \TpmCrude(\epsilon,p) = T(\epsilon/2,\epsilon/2,p)$, it holds  that
\begin{eqnarray*}
\w_t^{\top}\M\w_t = \sum_{i=1}^d\lambda_i(\w_t^{\top}\u_i)^2 \geq \sum_{i \in S(\epsilon/2)}(1-\epsilon/2)\lambda_1(\w_t^{\top}\u_i)^2 \geq (1-\epsilon/2)^2\lambda_1 > (1-\epsilon)\lambda_1.
\end{eqnarray*}

For the second part of the theorem, note that $S(\delta/\lambda_1) = \{1\}$. Thus we have that with probability at least $1-p$, for all $t\geq \TpmAcc(\lambda_1/\delta,\epsilon,p) = T(\delta/\lambda_1,\epsilon,p)$, it holds  that $(\w_t^{\top}\u_1) \geq 1-\epsilon$.

\end{proof}

\section{SVRG for Convex Functions given by Sums of Non-convex Functions and its Acceleration}\label{sec:svrg_nonconvex}

Suppose we want to minimize a function $F(x)$ that admits the following structure
\begin{eqnarray}\label{eq:svrgprob}
F(\x) = \frac{1}{n}\sum_{i=1}^nf_i(\x), 
\end{eqnarray}
where each $f_i$ is $\beta$ smooth, i.e.
\begin{eqnarray*}
\Vert{\nabla{}f_i(\x) - \nabla{}f_i(\y)}\Vert \leq \beta\Vert{\x-\y}\Vert \quad \forall{\x,\y\in\mathbb{R}^d} ,
\end{eqnarray*}
and $F(\x)$ is $\sigma$-strongly convex, i.e.
\begin{eqnarray*}
F(\y) \leq F(\x) + (\y-\x)^{\top}\nabla{}F(\x) + \frac{\sigma}{2}\Vert{\x-\y}\Vert^2 \quad \forall{\x,\y\in\mathbb{R}^d} .
\end{eqnarray*}

\begin{algorithm}[H]
\caption{\textsc{SVRG}}
\label{alg:svrg}
\begin{algorithmic}[1]
\STATE Input: $\tilde{\x}_0, \eta, m$
\FOR{$s=1,2,...$}
\STATE $\tilde{\x} \gets \tilde{\x}_{s-1}$
\STATE $\tilde{\bmu} \gets \nabla{}F(\tilde{\x})$
\STATE $\x_0 \gets \tilde{\x}$
\FOR{$t=1,2,...,m$}
\STATE Randomly pick $i_t\in[n]$
\STATE $\x_t \gets \x_{t-1} - \eta\left({\nabla{}f_{i_t}(\x_{t-1})-\nabla{}f_{i_t}(\tilde{\x}) + \tilde{\bmu}}\right)$ 
\ENDFOR
\STATE $\tilde{\x}_s \gets \frac{1}{m}\sum_{t=0}^{m-1}\x_t$
\ENDFOR
\end{algorithmic}
\end{algorithm}

\begin{theorem}\label{thm:svrg_nonconvex}
Suppose that each function $f_i(\x)$ in the objective \eqref{eq:svrgprob} is $\beta$-smooth and that $F(\x)$ is $\sigma$-strongly convex. Then for $\eta =\frac{\sigma}{7\beta^2}$ and $m \geq \frac{1}{2\eta^2\beta^2}$ it holds that
\begin{eqnarray*}
\E[\Vert{\tilde{\x}_s - \x^*}\Vert^2] \leq 2^{-s}\Vert{\tilde{\x}_0 - \x^*}\Vert^2 .
\end{eqnarray*}
\end{theorem}

\begin{proof}
We begin by analyzing the reduction in error on a single epoch $s$ and then apply the result recursively. Let us fix an iteration $t\in[m]$ of the inner loop in epoch $s$. In the sequel we denote by $\E_t[\cdot]$ the expectation with respect to the random choice of $i_t$ (i.e., the expectation is conditioned on all randomness introduced up to the $t$th iteration of the inner loop during epoch $s$).
Define 
\begin{eqnarray*}
\v_t = \nabla{}f_{i_t}(\x_{t-1}) - \nabla{}f_{i_t}(\tilde{\x}) + \tilde{\bmu}.
\end{eqnarray*}
Note that $\E_t[\v_t] = \nabla{}F(\x_{t-1})$ and thus $\v_t$ is an unbiased estimator for $\nabla{}F(\x_{t-1})$. We continue to upper bound the variance of $\v_t$ in terms of the distance of $\x_{t-1}$ and $\tilde{\x}$ from $\x^*$.
\begin{eqnarray*}
\E_t[\Vert{\v_t}\Vert^2] &\leq & 2\E_t[\Vert{\nabla{}f_{i_t}(\x_{t-1}) - \nabla{}f_{i_t}(\x^*)}\Vert^2] + 2\E_t[\Vert{\nabla{}f_{i_t}(\tilde{\x}) - \nabla{}f_{i_t}(\x^*) - \nabla{}F(\tilde{\x})}\Vert^2] \\
&=& 2\E_t[\Vert{\nabla{}f_{i_t}(\x_{t-1}) - \nabla{}f_{i_t}(\x^*)}\Vert^2] + 2\E_t[\Vert{\nabla{}f_{i_t}(\tilde{\x}) - \nabla{}f_{i_t}(\x^*)}\Vert^2] \\
&-& 4\nabla{}F(\tilde{\x})^{\top}(\nabla{}F(\tilde{\x}) - \nabla{}F(\x^*)) + 2\Vert{\nabla{}F(\tilde{\x})}\Vert^2 \\
&\leq & 2\E_t[\Vert{\nabla{}f_{i_t}(\x_{t-1}) - \nabla{}f_{i_t}(\x^*)}\Vert^2] + 2\E_t[\Vert{\nabla{}f_{i_t}(\tilde{\x}) - \nabla{}f_{i_t}(\x^*)}\Vert^2]  \\
&\leq & 2\beta^2(\Vert{\x_{t-1}-\x^*}\Vert^2 + \Vert{\tilde{\x}-\x^*}\Vert^2),
\end{eqnarray*}
where the first inequality follows from $(a+b)^2 \leq 2a^2+2b^2$, the first equality follows since $\E_t[f_{i_t}(\tilde{\x})]=\nabla{}F(\tilde{\x})$ (same goes for $\x^*$), the second inequality follows since $\nabla{}F(\x^*)=0$, and the third inequality follows from smoothness of $f_{i_t}$.

We now have that,
\begin{eqnarray*}
\E_t[\Vert{\x_t - \x^*}\Vert^2] &=& \Vert{\x_{t-1} - \x^*}\Vert^2 - 2\eta(\x_{t-1}-\x^*)^{\top}\E_t[\v_t] + \eta^2\E_t[\Vert{\v_t}\Vert^2] \\
& \leq & \Vert{\x_{t-1} - \x^*}\Vert^2 - 2\eta(\x_{t-1}-\x^*)^{\top}\nabla{}F(\x_{t-1}) + 2\eta^2\beta^2(\Vert{\x_{t-1}-\x^*}\Vert^2 + \Vert{\tilde{\x}-\x^*}\Vert^2) \\
&\leq & \Vert{\x_{t-1} - \x^*}\Vert^2 - 2\eta\sigma\Vert{\x_{t-1} - \x^*}\Vert^2 + 2\eta^2\beta^2(\Vert{\x_{t-1}-\x^*}\Vert^2 + \Vert{\tilde{\x}-\x^*}\Vert^2), 
\end{eqnarray*}
where the second inequality follow from convexity and strong-convexity of $F$.

Thus we have that,
\begin{eqnarray*}
\E[\Vert{\x_t - \x^*}\Vert^2] - \E[\Vert{\x_{t-1} - \x^*}\Vert^2] \leq 2\eta(\eta\beta^2-\sigma)\E[\Vert{\x_{t-1}-\x^*}\Vert^2] + 2\eta^2\beta^2\E[\Vert{\tilde{\x}-\x^*}\Vert^2] .
\end{eqnarray*}

Summing over all iterations of the inner loop on epoch $s$ we have
\begin{eqnarray*}
\E[\Vert{\x_m - \x^*}\Vert^2] - \E[\Vert{\x_0 - \x^*}\Vert^2] \leq 2\eta(\eta\beta^2-\sigma)\sum_{t=1}^m\E[\Vert{\x_{t-1}-\x^*}\Vert^2] + 2m\eta^2\beta^2\E[\Vert{\tilde{\x} - \x^*}\Vert^2] .
\end{eqnarray*}

Rearranging and using $\x_0 = \tilde{\x}$ we have that,
\begin{eqnarray*}
2\eta(\sigma-\eta\beta^2)\sum_{t=1}^m\E[\Vert{\x_t-\x^*}\Vert^2] \leq (1+2m\eta^2\beta^2)\E[\Vert{\tilde{\x}-\x^*}\Vert^2] .
\end{eqnarray*}

Using $\tilde{\x} = \tilde{\x}_{s-1}$ and $\tilde{\x}_s = \frac{1}{m}\sum_{t=0}^{m-1}\x_t$ we have that,
\begin{eqnarray*}
\E[\Vert{\tilde{\x}_s-\x^*}\Vert^2] \leq \frac{1+2m\eta^2\beta^2}{2\eta{}m(\sigma-\eta\beta^2)}\E[\Vert{\tilde{\x}_{s-1}-\x^*}\Vert^2] .
\end{eqnarray*}
Plugging the values of $\eta,m$ gives the theorem.
\end{proof}

\subsection{Acceleration of Algorithm \ref{alg:svrg}}

We now discuss how using the recent generic acceleration framework of \cite{Harchaoui15}, we can further accelerate Algorithm \ref{alg:svrg}. Note that Algorithm \ref{alg:svrg} requires to compute overall $\tilde{O}(n+m)=\tilde{O}\left({\frac{\beta^2}{\sigma^2}+n}\right)$ gradients of functions from the set $\lbrace{f_i(\x)}\rbrace_{i=1}^n$. Using the framework of \cite{Harchaoui15}, this quantity could be dramatically reduced.

On a very high-level, the framework of \cite{Harchaoui15} applies a convex optimization algorithm in an almost black-box fashion in order to simulate an algorithm known as the \textit{Accelerated Proximal-point} algorithm. That is, it uses the convex optimization algorithm to find an approximated global minimizer of the modified function:
\begin{eqnarray}\label{eq:AccSvrgProb}
\tilde{F}(\x) = \frac{1}{n}\sum_{i=1}^nf_i(\x) + \frac{\kappa}{2}\Vert{\x-\y}\Vert^2, 
\end{eqnarray}
where $\kappa, \y$ are parameters.

Note that the SVRG algorithm (\ref{alg:svrg}) could be directly applied to minimize \eqref{eq:AccSvrgProb} by considering the set of functions $\tilde{f}_i(\x) = f_i(\x) + \frac{\kappa}{2}\Vert{\x-\y}\Vert^2$ for all $i\in[n]$. It clearly holds that $\tilde{F}(\x) = \sum_{i=1}^n\tilde{f}_i(\x)$. Not also that now $\tilde{F}(\x)$ is $\sigma+\kappa$ strongly convex and for each $i\in[n]$ it holds that $\tilde{f}_i(\x)$ is $\beta+\kappa$ smooth.

The following Theorem (rephrased for our needs) is proven in \cite{Harchaoui15} (Theorem 3.1).
\begin{theorem}\label{thm:AccSVRGconv}
Fix the parameter $\kappa$. There exists an acceleration scheme for Algorithm \ref{alg:svrg} that finds an $\epsilon$-approximated minimizer of \eqref{eq:svrgprob}, after approximately minimizing $\tilde{O}\left({\sqrt{\frac{\sigma+\kappa}{\sigma}}}\right)$ instances of \eqref{eq:AccSvrgProb}.
\end{theorem}

Plugging Theorems \ref{thm:svrg_nonconvex}, \ref{thm:AccSVRGconv} and the proprieties of $\tilde{F}(\x)$ we have that Algorithm \ref{alg:svrg} could be applied to finding an $\epsilon$ minimizer of \eqref{eq:svrgprob} in total time:
\begin{eqnarray*}
\tilde{O}\left({\sqrt{\frac{\sigma+\kappa}{\sigma}}\left({T_G+\left({\frac{\beta+\kappa}{\sigma+\kappa}}\right)^2T_g}\right)}\right),
\end{eqnarray*}
where $T_G$ denotes the time to evaluate the gradient of $F$ and $T_g$ denotes the worst case time to evaluate the gradient of a single function $f_i$.

By optimizing the above bound with respect to $\kappa$ we arrive at the following theorem.

\begin{theorem}\label{thm:AccSVRGoptimized}
Assume that the gradient vector of each function $f_i(\x)$ could be computed in $O(d)$ time. Assume further that $\beta =\Omega(\sqrt{T_G/T_g}\sigma)$. Algorithm \ref{alg:svrg} combined with the acceleration framework of \cite{Harchaoui15}, finds an $\epsilon$-approximated minimizer of \eqref{eq:svrgprob} in total time $\tilde{O}\left({\sqrt{\frac{\beta}{\sigma}}T_G^{3/4}T_g^{1/4}}\right)$.
\end{theorem}

\end{document}